\documentclass{BDFT-article-class}

\usepackage[utf8]{inputenc}
\usepackage[T1]{fontenc}
\usepackage[english]{babel}
\usepackage{mathtools,amsmath,amssymb,amsfonts,amsthm}
\usepackage{bbm}
\mathtoolsset{showonlyrefs}
\usepackage{enumitem}

\usepackage{geometry}
\usepackage{color} 

\usepackage[pdftex,bookmarks]{hyperref}

\newtheorem{theorem}{Theorem}[section]
\newtheorem{lemma}{Lemma}[section]
\newtheorem{proposition}{Proposition}[section]
\newtheorem{corollary}{Corollary}[section]

\theoremstyle{definition}
\newtheorem{remark}{Remark}[section]

\numberwithin{equation}{section}

\title[Parabolic orbits in Celestial Mechanics]{Parabolic orbits in Celestial Mechanics: \\a functional-analytic approach}

\author[A.~Boscaggin]{Alberto~Boscaggin}
\address{Dipartimento di Matematica ``Giuseppe Peano'', Universit\`{a} di Torino \\
Via Carlo Alberto, 10 - 10123 Torino, Italy}
\email{alberto.boscaggin@unito.it}

\author[W.~Dambrosio]{Walter~Dambrosio}
\address{Dipartimento di Matematica ``Giuseppe Peano'', Universit\`{a} di Torino \\
Via Carlo Alberto, 10 - 10123 Torino, Italy}
\email{walter.dambrosio@unito.it}

\author[G.~Feltrin]{Guglielmo~Feltrin}
\address{Dipartimento di Scienze Matematiche, Informatiche e Fisiche, Università di Udine \\
Via delle Scienze, 206 - 33100 Udine, Italy}
\email{guglielmo.feltrin@uniud.it}

\author[S.~Terracini]{Susanna~Terracini}
\address{Dipartimento di Matematica ``Giuseppe Peano'', Universit\`{a} di Torino \\
Via Carlo Alberto, 10 - 10123 Torino, Italy}
\email{susanna.terracini@unito.it}

\thanks{Work supported by the ERC Advanced Grant 2013 n.~339958 \textit{Complex Patterns for Strongly Interacting Dynamical Systems - COMPAT}. Work written under the auspices of the Grup\-po Na\-zio\-na\-le per l'Anali\-si Ma\-te\-ma\-ti\-ca, la Pro\-ba\-bi\-li\-t\`{a} e le lo\-ro Appli\-ca\-zio\-ni (GNAMPA) of the Isti\-tu\-to Na\-zio\-na\-le di Al\-ta Ma\-te\-ma\-ti\-ca (INdAM).
\\
\textbf{Preprint -- August 2020}} 

\begin{document}



\date{}

\begin{abstract}
We prove the existence of half-entire parabolic solutions, asymptotic to a prescribed central configuration, for the equation
\begin{equation*}
\ddot{x} = \nabla U(x) + \nabla W(t,x), \qquad x \in \mathbb{R}^{d},
\end{equation*}
where $d \geq 2$, $U$ is a positive and positively homogeneous potential with homogeneity degree $-\alpha$ with $\alpha\in\mathopen{]}0,2\mathclose{[}$, and $W$ is a (possibly time-dependent) lower order term, for $\vert x \vert \to +\infty$, with respect to $U$. The proof relies on a perturbative argument, after an appropriate formulation of the problem in a suitable functional space. Applications to several problems of Celestial Mechanics (including the $N$-centre problem, the $N$-body problem and the restricted $(N+H)$-body problem) are given.
\end{abstract}

\maketitle

\section{Introduction and statement of the main result}\label{section-1}

In this paper, we are concerned with systems of second-order ODEs of the type
\begin{equation}\label{eq-x}
\mu_{j} \ddot{x}_{j} = \partial_{x_{j}} U(x) + \partial_{x_{j}} W(t,x), \qquad j = 1,\ldots,d,
\end{equation}
where $x = (x_1,\ldots,x_d) \in \mathbb{R}^d$, with $d \geq 2$, and:
\begin{itemize}
\item $\mu_1,\ldots,\mu_d > 0$, 
\item $U \in \mathcal{C}^{3}(\Sigma,\mathopen{]}0,+\infty\mathclose{[})$ is a positively homogeneous potential with homogeneity degree $-\alpha$ with $\alpha\in\mathopen{]}0,2\mathclose{[}$, where $\Sigma \subseteq \mathbb{R}^d \setminus \{0\}$ is an open set such that, if $x \in \Sigma$ and $\lambda > 0$, then $\lambda x \in \Sigma$, 
\item $W$ is a lower order term, for $\lvert x \rvert \to +\infty$, with respect to $U$ (the precise condition will be given in the statement of the main result).
\end{itemize}
In vector notation, equation \eqref{eq-x} can be written as
\begin{equation}\label{eq-main}
\ddot x = \nabla U(x) + \nabla W(t,x),
\end{equation}
where the symbol $\nabla$ stands for the gradient with respect to the so-called mass scalar product
\begin{equation}\label{eq-scalarepesato}
\langle x, y \rangle = \sum_{j=1}^{d} \mu_{j} x_{j} y_{j}, \qquad x,y\in\mathbb{R}^{d}.
\end{equation}
This notation will be used throughout the paper; we will also write $\lvert x \rvert = \sqrt{\langle x, x \rangle}$
for the associated norm.

Equation \eqref{eq-main} is motivated by problems of Celestial Mechanics 
(having $\alpha = 1$, by Newton's law of gravitation).
As a first example, when $\mu_{j} = 1$ for every $j$ and $U(x) = \nu/\|x\|$ (here, $\| \cdot \|$ stands for the Euclidean norm of a $d$-dimensional vector), \eqref{eq-main} reduces to the equation
\begin{equation}\label{eq-kepler-intro}
\ddot x = - \frac{\nu x}{\| x \|^3} + \nabla W(t,x),
\end{equation}
which can be meant as a (possibly time-dependent) perturbation of the Kepler problem in the $d$-dimensional space.
For instance, both the classical $N$-centre problem and the elliptic restricted three-body problem can be written in this form, with
$W(t,x) = W(x)$ in the former case and $W(t,x)$ periodic in time in the latter one (cf.~Section \ref{section-6.2})
As a further application, we can take $d = kN$ (with $k,N \geq 2$), 
\begin{equation*}
(\mu_1,\ldots,\mu_d) = (\overbrace{m_1,\ldots,m_1}^{\text{$k$ times}},m_2,\ldots,m_2,\ldots,m_N,\ldots,m_N)
\end{equation*}
and 
\begin{equation*}
U(x) = \sum_{i<j}\frac{m_i m_{j}}{\|q_i - q_{j}\|},
\qquad
\text{$x = (q_1,\ldots,q_N)$, with $q_i \in \mathbb{R}^k$,}
\end{equation*}
(here, $\| \cdot \|$ stands for the Euclidean norm of a $k$-dimensional vector) so that \eqref{eq-main} yields
\begin{equation}\label{eq-N-intro}
m_{i} \ddot q_{i} = -\sum_{j \neq i} \frac{m_{i} m_{j} (q_{i} - q_{j})}{\|q_{i} - q_{j} \|^3} + \frac{\partial}{\partial q_{i}} W(t,q_1,\ldots,q_N), \qquad i=1,\ldots,N,
\end{equation}
that is, a perturbation of the classical $N$-body problem in the $k$-dimensional space (cf.~Section \ref{section-6.1}).
Incidentally, notice that in such a case the potential $U$ is defined only when
$q_i \neq q_{j}$ for every $i \neq j$, so that $\Sigma \subsetneq \mathbb{R}^{kN} \setminus \{0\}$.
More general situations could also be treated (cf.~Section~\ref{section-6.3}).

Our interest is in constructing half-entire solutions $x \colon \mathopen{[}0,+\infty\mathclose{[} \to \mathbb{R}^{d}$ to \eqref{eq-main} approaching infinity with zero velocity, namely
\begin{equation*}
\lim_{t \to +\infty} \lvert x(t) \rvert = +\infty \quad \text{ and } \quad \lim_{t \to +\infty} \dot{x}(t) = 0.
\end{equation*}
Throughout the paper, we will call such solutions \emph{parabolic}, according to the terminology used by Chazy in its pioneering paper \cite{Ch-22} investigating all the possible final states for a three-body problem as time goes to infinity. Incidentally, let us notice that, in the elementary case of the unperturbed Kepler problem (that is, equation \eqref{eq-kepler-intro} with $W = 0$) such solutions indeed lie on Keplerian parabolas. 

In the last decades, parabolic orbits for various equations of Celestial Mechanics have been investigated by many authors, from different point of views. For the circular restricted three body-problem, their existence was first proved by proved by McGehee \cite{Mc-73} via Dynamical Systems techniques: indeed, via a suitable change of variables, ``infinity'' (with zero velocity) can be regarded as a fixed point of a suitable Poincar\'{e} map, so that tools from the topological theory of invariant manifolds apply
(see \cite{BaFoMa-20} for an updated bibliography about this line of research). More recent contributions deal with the existence of parabolic solutions for $N$-centre and $N$-body problems \cite{BaHuPoTe-PP,BaTeVe-13,BaTeVe-14,BoDaPa-18,BoDaTe-17,LuMa-14,MaVe-09,MoMoSa-18,Mo-PP,Sa-84,SaHu-81}, often in connection with the scattering problem. Generally speaking, the interest for this kind of orbits
mainly comes from the fact that, in spite of the natural intrinsic instability, they can be used as carriers from different regions of the configuration space and, eventually, as building blocks in the construction of solutions with chaotic behavior 
(see \cite{GuMaSe-16,SoTe-12} and the references therein). Moreover, parabolic solutions are known to provide precious information on the behavior of general solutions near collisions \cite{Ch-98,Mc-74,Mo-89}; finally, they play a role in the applications of weak KAM theory to Celestial Mechanics and can be used to construct weak KAM solutions of the associated Hamilton--Jacobi equation \cite{FaMa-07,Ma-12}.

In more details, here we focus on the possibility of constructing a parabolic orbit starting from a prescribed configuration $x_0 \in \mathbb{R}^d$ at $t = 0$ and having a prescribed asymptotic direction,
that is, $\lim_{t \to +\infty}x(t)/|x(t)| = \xi^+$. Incidentally, let us recall that, under quite general assumptions (cf.~\cite{BaTeVe-14,Ch-98,SaHu-81}), whenever this limit exists, then $\xi^+$ must be a (normalized) central configuration of the potential $U$.
In the case of the $N$-body problem (that is, equation \eqref{eq-N-intro} with $W = 0$), such an issue has been addressed in the pioneering paper by Maderna and Venturelli \cite{MaVe-09}. More precisely, they constructed, with variational arguments, parabolic solutions starting from an arbitrary configuration and approaching, at infinity, any prescribed minimizing central configuration. 
Such solutions, having infinite action, were obtained as
limits of solutions of approximating two-point boundary value problems; the existence of the approximate solutions was ensured by the direct method of the Calculus of Variations (together with Marchal's lemma), while delicate action level estimates, strongly relying on the homogeneity of the $N$-body problem, were then used to show the convergence to a limit solution and its parabolicity. 

In this paper, we are going to prove a perturbative version of this result (meaning that the starting configuration will not be arbitrary, but rather on a conic neighborhood of the central configuration), which however enhances it as for two different aspects. First, the minimality of the central configuration will not be assumed; instead, it will be replaced by a spectral condition introduced in the paper by Barutello and Secchi \cite{BaSe-08} and later extensively explored in \cite{BaHuPoTe-PP}. Second, our result will be valid not only for the $N$-body problem, but in the much more general setting of equation \eqref{eq-main}. In this regard, the fact the perturbation term $W$ is allowed to be \textit{time-dependent} has substantial consequences, this being indeed, as well known, the source of a more complex dynamics. For instance, while for autonomous problem (e.g., $N$-body, $N$-centre) parabolic solutions to \eqref{eq-main} exiting a large ball are forced to go to infinity (as a consequence of the Lagrange-Jacobi inequality, see \cite[p. 94]{Ch-98} and \cite[Lemma~2.1]{BoDaTe-17}), when $W$ depends explicitly on time solutions with oscillatory behavior (that is, $\liminf_{t \to +\infty}|x(t)| < \limsup_{t \to +\infty}|x(t)| = +\infty$) shadowing parabolic orbits, may exist (see \cite{GuMaSe-16,GuSeMaSa-17,LlSi-80} and the references therein). For these reasons, it seems to be an hard task to construct parabolic solutions via an approximation argument similar to the one in \cite{MaVe-09} and we will indeed follow a completely different approach. 

To state our result precisely, we introduce the so-called inertia ellipsoid
\begin{equation*}
\mathcal{E} = \bigl{\{} x \in \mathbb{R}^d \colon |x| = 1 \bigr{\}}
\end{equation*}
and we set
\begin{equation*}
\mathcal{U} = U \vert_\mathcal{E}.
\end{equation*}
As usual, we say that $\xi^+ \in \mathcal{E}$ is a (normalized) \emph{central configuration} if it is a critical point of $\mathcal{U}$. with this in mind, the Barutello-Secchi spectral condition reads as follows:
\begin{itemize}[leftmargin=30pt,labelsep=8pt,topsep=5pt]
\item[$(\textsc{BS})$] \emph{the smallest eigenvalue $\nu_{1}$ of $\nabla^{2} \mathcal{U}(\xi^{+})$ satisfies}
\begin{equation*}
\nu_{1} > -\dfrac{(2-\alpha)^{2}}{8} \, \mathcal{U}(\xi^{+}),
\end{equation*}
\end{itemize}
where $\nabla^{2}$ is the Hessian matrix of the function $\mathcal{U}$ with respect to the metric \eqref{eq-scalarepesato}
(hence, $\nabla^2 = M^{-1} D^{2}$, where $D^{2}$ denotes the Hessian with respect to the Euclidean metric and $M=\textnormal{diag}(\mu_1,\ldots,\mu_d)$). Notice that when $\xi^+$ is a minimizing central configuration (that is, a global minimum of $\mathcal{U}$), then the $(\textsc{BS})$-condition is of course satisfied.

We finally define, for $R > 0$ and $\eta\in\mathopen{]}0,1\mathclose{[}$, the set
\begin{equation}\label{cone}
\mathcal{T}(\xi^{+},R,\eta) = \biggl{\{} x \in \mathbb{R}^d \colon \lvert x \rvert > R \text{ and }  \Bigl{\langle} \dfrac{x}{\lvert x \rvert}, \xi^{+} \Bigr{\rangle} > \eta \biggr{\}}.
\end{equation}
Our main result is the following (see the end of this introduction for some clarification about the notation used). 

\begin{theorem}\label{th-main}
Let $d \in \mathbb{N}$, with $d\geq 2$, and let $\mu_1,\ldots,\mu_d > 0$.
Let $U \in \mathcal{C}^{3}(\Sigma,\mathopen{]}0,+\infty\mathclose{[})$ be a $-\alpha$-homogeneous potential, with $\alpha\in\mathopen{]}0,2\mathclose{[}$, and let 
$\xi^+ \in \mathcal{E}$ be a central configuration for $U$ satisfying the $(\textsc{BS})$-condition.
Finally, let $W\in\mathcal{C}^{2}(\mathopen{[}0,+\infty\mathclose{[}\times \mathcal{T}(\xi^{+},R,\eta))$ for some $R>0$ and $\eta\in\mathopen{]}0,1\mathclose{[}$ satisfy
\begin{equation}\label{hp-main}
\lvert W(t,x) \rvert + \lvert x \rvert \lvert \nabla W(t,x) \rvert + \lvert x \rvert^{2} \lvert \nabla^{2} W(t,x) \rvert = \mathcal{O}\bigl{(}\lvert x \rvert^{-\beta} \bigr{)},
\end{equation}
for $\lvert x \rvert \to +\infty$, uniformly in $t$, for some $\beta$ such that $4\beta-3\alpha>2$.
Then, there exist $R' \geq R$ and $\eta' \in \mathopen{[}\eta,1\mathclose{[}$ such that $\mathcal{T}(\xi^{+},R',\eta') \subseteq \Sigma$ and for every $x_{0} \in \mathcal{T}(\xi^{+},R',\eta')$ there exists a parabolic solution $x \colon \mathopen{[}0,+\infty\mathclose{[} \to \mathbb{R}^{d}$ of equation \eqref{eq-main} satisfying $x(0) = x_{0}$ and 
\begin{equation*}\lim_{t\to+\infty} \dfrac{x(t)}{\lvert x(t) \rvert} = \xi^{+}.
\end{equation*}
Moreover, $\lvert x(t) \rvert \sim \omega t^{\frac{2}{\alpha+2}}$ for $t \to +\infty$, where $\omega=\bigl{(}(\alpha+2)^{2} U(\xi^{+}) /2 \bigr{)}^{\!\frac{1}{\alpha+2}}$.
\end{theorem}

In order to prove Theorem~\ref{th-main}, we face the problem from a functional analytical point of view, applying the implicit function theorem in a suitable space of functions defined on the half line: the required non-degeneracy will be ensured via the use of an Hardy-type inequality \cite{HaLi}. As a consequence, our parabolic solutions turn out to be local minimizers of a renormalized action functional, in the sense of Remark~\ref{rem-3.1}; this variational characterization makes these orbits suitable for the analysis of their Maslov indices as in \cite{BaHuPoTe-PP}. 

Our proof can be roughly described as follows. The crucial idea is to look for solutions to \eqref{eq-main} having the form
\begin{equation*}
x(t) = x_{0}(t + \lambda) + u_{\lambda}(t),
\end{equation*}
where $\lambda > 0$ is a (large) parameter, $x_{0}$ is the homotethic parabolic solution with direction $\xi^{+}$ of the unperturbed problem
$\ddot x = \nabla U(x)$ (that is, $x_{0}(t) = \omega t^{\frac{2}{\alpha + 2}} \xi^{+}$) and $u_{\lambda}$ is a perturbation term, lying in a suitable functional space $X$.
Quite unexpectedly, a proper choice of $X$ can be made (essentially, $X$ is a space of functions with $L^{2}$-weak derivative, see Section~\ref{section-3} for more details), in order for the problem in the new unknown $u_{\lambda}$ to have a standard variational structure (despite the equation being considered on a non-compact time interval) and for $x(t) = x_{0}(t+\lambda) + u_{\lambda}(t)$ to satisfy the desired asymptotic properties (namely, $\dot{x}(t) \to 0$ and $x(t)/\lvert x(t) \rvert \to \xi^{+}$) whenever $u_{\lambda} \in X$ is a solution.
Even more, the problem for $u_{\lambda}$ can be solved by elementary perturbation arguments:
essentially, taking $\lambda \to +\infty$ equation \eqref{eq-main} reduces to the unperturbed problem
$\ddot x = \nabla U(x)$ and, due to the validity of a Hardy-type inequality in $X$, the implicit function theorem can be easily applied. Of course, a major drawback of this procedure lies in its purely perturbative nature: as a consequence, in Theorem~\ref{th-main} we need to assume that 
the initial condition $x_{0}$ lies in the set $\mathcal{T}(\xi^{+},R',\eta')$ defined in \eqref{cone}, for suitable $R'$ and $\eta'$.

In a forthcoming paper \cite{BDFT-21}, we will show how to use Theorem~\ref{th-main} as a tool in the search of entire parabolic and capture solutions in the restricted $(N+1)$-body problem.

\medskip

The plan of the paper is the following. In Section~\ref{section-2} we collect some simple preliminary formulas concerning positively homogeneous potentials and their central configurations. In Section~\ref{section-3} we give the definition of the functional space $X = \mathcal{D}^{1,2}_{0}(t_{0},+\infty)$, and we present some of its properties. In Section~\ref{section-4} we describe in more detail the general strategy (just very briefly sketched in the above discussion), so as to provide an outline of the proof of Theorem~\ref{th-main}; the complete proof is then given in Section~\ref{section-5}. Finally, in Section~\ref{section-6} we discuss some applications of the main result to probems of Celestial Mechanics.

\medskip

We end this introductory part by presenting some symbols and some notation used in the present paper. 
As already mentioned, by $\langle\cdot,\cdot\rangle$ and $\lvert\cdot\rvert$ we mean the mass scalar product \eqref{eq-scalarepesato} in $\mathbb{R}^d$ and the associated norm, respectively. Accordingly, $\nabla$ and $\nabla^2$ denotes the gradient and the Hessian matrix of a function with respect to this metric. The symbol $\lvert\cdot\rvert$ is used also for the operator norm of a matrix $A \in \mathbb{R}^{d \times d}$, that is, 
$\lvert A \rvert = \sup_{\lvert x \rvert \leq 1}\lvert Ax \rvert$.

\section{Preliminaries on homogeneous functions and central configurations}\label{section-2}

In this section we recall some well-known facts concerning positively homogeneous potentials and the related central configurations.

Let $\Sigma \subseteq \mathbb{R}^{d}$ be an open set such that $\lambda x \in \Sigma$ for every $x\in \Sigma$ and $\lambda> 0$. Let $U \colon \Sigma \to \mathbb{R}$ be a positively homogeneous function of degree $-\alpha<0$,
namely
\begin{equation}\label{eq-hom1}
U(\lambda x) = \lambda^{-\alpha} U(x),  \quad \text{for every $x\in \Sigma$ and $\lambda> 0$.}
\end{equation}
Throughout this section, we assume that $U$ is of class $\mathcal{C}^{2}$.
By differentiating \eqref{eq-hom1} with respect to $\lambda$, we have
\begin{equation*}
\langle \nabla U(\lambda x), x \rangle = -\alpha \, \lambda^{-\alpha-1} U(x), 
\quad \text{for every $x\in \Sigma$ and $\lambda> 0$,}
\end{equation*}
and so, for $\lambda=1$ we obtain the Euler's homogeneous function formula
\begin{equation}\label{Euler-formula}
\langle \nabla U(x),x \rangle = -\alpha \, U(x), \quad \text{for every $x\in \Sigma$.}
\end{equation}

By differentiating \eqref{eq-hom1} with respect to $x$ and dividing by $\lambda$, we deduce
\begin{equation*}
\nabla U(\lambda x) = \lambda^{-\alpha-1} \nabla U(x), \quad \text{for every $x\in \Sigma$ and $\lambda> 0$,}
\end{equation*}
and, by differentiating again,
\begin{equation*}
\nabla^{2} U(\lambda x) = \lambda^{-\alpha-2} \nabla^{2} U(x), \quad \text{for every $x\in \Sigma$ and $\lambda> 0$.}
\end{equation*}
These two formulas mean that $\nabla U$ and $\nabla^{2} U$ are positively homogeneous functions of degrees $-\alpha-1$ and $-\alpha-2$, respectively.

By differentiating \eqref{Euler-formula} with respect to $x$, we obtain
\begin{equation*}
\nabla^2 U(x) x+\nabla U(x) = - \alpha \, \nabla U(x), \quad \text{for every $x\in\Sigma$,}
\end{equation*}
and so
\begin{equation}\label{eq-diffsecom}
\nabla^2 U(x) x=-(\alpha+1) \nabla U(x), \quad \text{for every $x\in\Sigma$.}
\end{equation}

Let $\xi^{+} \in \mathcal{E}$ be a (normalized) central configuration. From \eqref{Euler-formula} we deduce that $\langle \nabla U(\xi^{+}),\xi^{+} \rangle = -\alpha  \, U(\xi^{+})$ and, as a consequence of the Lagrange multipliers theorem, we have
\begin{equation} \label{eq-diffprimocc}
\nabla U(\xi^{+}) = - \alpha \, U(\xi^{+}) \xi^{+}.
\end{equation}
By using \eqref{eq-diffprimocc} in \eqref{eq-diffsecom}, we infer that
\begin{equation}\label{eq-diffseccc}
\nabla^2 U(\xi^{+}) \xi^{+} = \alpha (\alpha+1) U(\xi^{+}) \xi^{+}.
\end{equation}
Formula \eqref{eq-diffseccc} will have a crucial role in Section~\ref{section-5.2}.

\section{The space $\mathcal{D}^{1,2}_{0}(t_{0},+\infty)$}\label{section-3}

As described in the introduction, a crucial role in our proof is played by the choice of the functional space. Indeed, after some change of variables we will be led to work in the space 
$\mathcal{D}^{1,2}_{0}(t_{0},+\infty)$,
with $t_{0}>0$, defined as the subset of vector-valued continuous functions $\varphi \colon \mathopen{[}t_{0},+\infty\mathclose{[ \to \mathbb{R}^d}$ which vanish at $t=t_{0}$ and can be written as primitives of $L^2$-functions. Precisely,
\begin{equation*}
\mathcal{D}^{1,2}_{0}(t_{0},+\infty) = \biggl{\{} \varphi \in \mathcal{C}(\mathopen{[}t_{0},+\infty\mathclose{[}) \colon \varphi(t) = \int_{t_{0}}^{t} v(s)\,\mathrm{d}s \text{ for some } v \in L^{2}(t_{0},+\infty)\biggr{\}}.
\end{equation*}
Readers familiar with the theory of Sobolev spaces (see, for instance, \cite[Chapter~8]{Br-11}) will immediately observe that any $\varphi \in \mathcal{D}^{1,2}_{0}(t_{0},+\infty)$ is differentiable in the sense of distributions, with $\dot{\varphi} = v$; conversely, locally integrable functions with distributional derivatives in $L^{2}(t_{0},+\infty)$ belong to $\mathcal{D}^{1,2}_{0}(t_{0},+\infty)$ as long as $\varphi(t_{0}) = 0$ (recall that continuity up to $t =t_{0}$ is automatically ensured whenever the weak derivative is in $L^{p}(t_{0},+\infty)$ for some $p \in \mathopen{[}1,+\infty\mathclose{]}$). Let us also observe that, by the Cauchy--Schwartz inequality, the pointwise estimate
\begin{equation}\label{Hardy0}
\lvert \varphi(t) \rvert  \leq \Vert \dot{\varphi} \Vert_{L^{2}} \sqrt{t-t_{0}} \leq \Vert \dot{\varphi} \Vert_{L^{2}} \sqrt{t}, \quad \text{for every $t \geq t_{0}$},
\end{equation}
holds true, implying that any $\varphi \in \mathcal{D}^{1,2}_{0}(t_{0},+\infty)$ grows at infinity at most as $\sqrt{t}$.
We stress once more that, in formula \eqref{Hardy0}, the symbol $| \cdot |$ stands for the norm coming from the mass scalar product, and the $L^2$-norm of $\dot \varphi$ has to be meant accordingly.

In the following, we will endow $\mathcal{D}^{1,2}_{0}(t_{0},+\infty)$ with the norm
\begin{equation}\label{defnorm}
\Vert\varphi\Vert=\biggl{(} \int_{t_{0}}^{+\infty} |\dot{\varphi}(t)|^{2} \,\mathrm{d}t \biggr{)}^{\!\frac{1}{2}}.
\end{equation}
As a consequence, \eqref{Hardy0} writes as
\begin{equation}\label{Hardy-sqrt}
\lvert \varphi(t) \rvert \leq \Vert \varphi \Vert \sqrt{t}, \quad \text{for every $t \geq t_{0}$},
\end{equation}
implying that $\varphi_{n}(t) \to \varphi(t)$ uniformly in $\mathopen{[}t_{0},K\mathclose{]}$ for any $K > t_{0}$, whenever $\varphi_{n} \to \varphi$ in $\mathcal{D}^{1,2}_{0}(t_{0},+\infty)$.

\begin{proposition}\label{hilbert}
The space $\mathcal{D}^{1,2}_{0}(t_{0},+\infty)$ is a Hilbert space containing the set $\mathcal{C}^{\infty}_{\mathrm{c}}(\mathopen{]}t_{0},+\infty\mathclose{[})$ as a dense subspace.
\end{proposition}

\begin{proof}
We first observe that the norm \eqref{defnorm} is induced by the scalar product 
\begin{equation*}
(\varphi_{1},\varphi_{2}) \mapsto \int_{t_{0}}^{+\infty} \bigl{\langle} \dot{\varphi}_{1}(t), \dot{\varphi}_{2}(t) \bigr{\rangle} \,\mathrm{d}t. 
\end{equation*}
Hence, to prove that $\mathcal{D}^{1,2}_{0}$ is a Hilbert space (for the rest of the proof, for briefness we omit the domain $(t_{0},+\infty)$ since no ambiguity is possible) we just need to show that Cauchy sequences are convergent. To this end, let us consider a Cauchy sequence $(\varphi_{n})_{n} \subseteq \mathcal{D}^{1,2}_{0}$. By definition of the norm, it follows that $(\dot{\varphi}_{n})_{n}$ is a Cauchy sequence in $L^{2}$; therefore, there exists $v \in L^{2}$ such that $\dot{\varphi}_{n} \to v$ in $L^{2}$. Setting $\varphi(t) = \int_{t_{0}}^{t} v(s)\,\mathrm{d}s$, we clearly have $\varphi \in \mathcal{D}^{1,2}_{0}$ and $\varphi_{n} \to \varphi$ in $\mathcal{D}^{1,2}_{0}$, thus proving the completeness of $\mathcal{D}^{1,2}_{0}$.

We now show the second part of the statement. Given $\varphi \in \mathcal{D}^{1,2}_{0}$, let us take $(v_{n})_{n} \subseteq \mathcal{C}^{\infty}_{\mathrm{c}}(\mathopen{]}t_{0},+\infty\mathclose{[})$ such that $v_{n} \to \dot{\varphi}$ in $L^{2}$ and define 
$J_n = \mathrm{supp}(v_{n})$. Setting $T_{n} = \sup J_n$ and taking 
$\gamma \in \mathcal{C}^\infty(\mathopen{[}0,+\infty\mathclose{[})$ such that $\gamma(t) = 1$ for $t \in \mathopen{[}0,1\mathclose{]}$ and $\gamma(t) = 0$ for $t \geq 2$,
let us define, for $t \geq t_{0}$, 
\begin{equation*}
\gamma_{n}(t) = \gamma\biggl{(}\dfrac{t-t_{0}}{n T_{n}}\biggr{)}
\end{equation*}
and
\begin{equation*}
\varphi_{n}(t) = \gamma_{n}(t) \int_{t_{0}}^{t} v_{n}(s)\,\mathrm{d}s.
\end{equation*}
Clearly, $\varphi_{n} \in \mathcal{C}^{\infty}_{\mathrm{c}}(\mathopen{]}t_{0},+\infty\mathclose{[})$; we thus need to show that
$\dot{\varphi}_{n} \to \dot{\varphi}$ in $L^{2}$. To this end, we first observe that
\begin{equation}\label{phipunto}
\dot{\varphi}_{n}(t) = \gamma_{n}(t)v_{n}(t) + \frac{1}{n T_{n}}\dot{\gamma}\biggl{(}\dfrac{t-t_{0}}{n T_{n}}\biggr{)} \int_{t_{0}}^{t} v_{n}(s)\,\mathrm{d}s.
\end{equation}
The first term on the right-hand side converges to $\dot{\varphi}$ in $L^{2}$, since
\begin{align*}
\Vert \gamma_{n} v_{n} - \dot{\varphi} \Vert_{L^{2}} & \leq \Vert \gamma_{n} (v_{n} - \dot{\varphi}) \Vert_{L^{2}} + \Vert (\gamma_{n} - 1)\dot{\varphi}\Vert_{L^{2}}
\\ & \leq \Vert \gamma \Vert_{L^{\infty}} \Vert v_{n} - \dot{\varphi} \Vert_{L^{2}} + \biggl{(}\int_{t_{0}}^{+\infty} \lvert (\gamma_{n}(t)-1) \dot{\varphi}(t)\rvert^{2} \,\mathrm{d}t\biggr{)}^{\!\frac{1}{2}}
\end{align*}
and the integral goes to zero by the dominated convergence theorem. On the other hand, the second term on the right-hand side of \eqref{phipunto} goes to zero in $L^{2}$; indeed, since by the Cauchy--Schwartz inequality
\begin{equation*}
\biggl{\lvert} \int_{t_{0}}^{t} v_{n}(s)\,\mathrm{d}s \biggr{\rvert} \leq \sqrt{T_{n}} \, \Vert v_{n} \Vert_{L^{2}}, \quad \text{for every $t \geq t_{0}$},
\end{equation*}
it holds
\begin{align*}
\int_{t_{0}}^{+\infty} \biggl{\lvert} \dfrac{1}{n T_{n}}\dot{\gamma}\biggl{(}\dfrac{t-t_{0}}{ n T_{n}}\biggr{)} \int_{t_{0}}^{t} v_{n}(s)\,\mathrm{d}s \biggr{\rvert}^{2} \,\mathrm{d}t &
\leq \dfrac{\Vert v_{n} \Vert^{2}_{L^{2}}}{n^{2} T_{n}} \int_{t_{0}}^{+\infty} \biggl{\lvert} \dot{\gamma}\biggl{(}\dfrac{t-t_{0}}{n T_{n}}\biggr{)} \biggr{\rvert}^{2} \,\mathrm{d}t \\
& = \frac{\Vert v_{n} \Vert^{2}_{L^{2}}}{n} \int_{0}^{+\infty} \lvert \dot{\gamma}(s) \rvert^{2} \,\mathrm{d}s,
\end{align*}
which goes to zero as $n \to \infty$, since $\Vert v_{n} \Vert_{L^{2}} \to \Vert \dot{\varphi} \Vert_{L^{2}}$.
\end{proof}

\begin{remark}\label{rem-2.1}
The notation ``$\mathcal{D}^{1,2}_{0}$'' is borrowed from the theory of Sobolev spaces for functions in $\mathbb{R}^{N}$. Indeed, for $p \in \mathopen{[}1,N\mathclose{[}$ the space $\mathcal{D}^{1,p}(\mathbb{R}^{N})$ is defined as the completion of $\mathcal{C}_{\mathrm{c}}^\infty(\mathbb{R}^{N})$ with respect to the norm $\bigl{(} \int \lvert \nabla \varphi \rvert^{p} \bigr{)}^{\!\frac{1}{p}}$ (by the Sobolev--Gagliardo--Nirenberg inequality, it turns out that $\mathcal{D}^{1,p}(\mathbb{R}^{N}) \subset L^{p^{*}}(\mathbb{R}^{N})$, with $p^{*}$ the Sobolev critical exponent). Proposition~\ref{hilbert} thus shows that an analogous characterization can be given for $\mathcal{D}^{1,2}_{0}(t_{0},+\infty)$ (here, the subscript $0$ has been added to mean that functions vanish for $t = t_{0}$); however, a more direct and elementary construction seems to be preferable in the $1$-dimensional setting.
\end{remark}

As a consequence of Proposition~\ref{hilbert}, we can finally state and prove a further inequality, showing
that $\mathcal{D}^{1,2}_{0}(t_{0},+\infty)$ is continuously embedded in a weighted $L^{2}$-space. Precisely, we have the following classical Hardy-type inequality (cf.~\cite{HaLi}).

\begin{proposition}\label{prop-Hardy}
For every $\varphi \in \mathcal{D}^{1,2}_{0}(t_{0},+\infty)$, it holds that
\begin{equation}\label{Hardy}
\int_{t_{0}}^{+\infty} \dfrac{\lvert \varphi(t) \rvert^{2}}{t^{2}} \,\mathrm{d}t \leq 4 \int_{t_{0}}^{+\infty} \lvert\dot{\varphi}(t)\rvert^{2} \,\mathrm{d}t.
\end{equation}
\end{proposition}

\begin{proof}
Let us first assume that $\varphi \in \mathcal{C}^{\infty}_{\mathrm{c}}(\mathopen{]}t_{0},+\infty\mathclose{[})$. In this case, integrating by parts
and using the Cauchy--Schwartz inequality, we directly obtain
\begin{align*}
& \int_{t_{0}}^{+\infty} \dfrac{\lvert\varphi(t) \rvert^{2}}{t^{2}} \,\mathrm{d}t 
= -\int_{t_{0}}^{+\infty} \frac{\mathrm{d}}{\mathrm{d}t} \biggl{(}\dfrac{1}{t}\biggr{)}\lvert\varphi(t) \rvert^{2} \,\mathrm{d}t  = 2 \int_{t_{0}}^{+\infty} \frac{\langle \varphi(t),\dot\varphi(t) \rangle}{t}\,\mathrm{d}t \\
& \leq  \int_{t_{0}}^{+\infty} \dfrac{2\lvert\varphi(t) \rvert \lvert\dot{\varphi}(t) \rvert}{t} \,\mathrm{d}t  
\leq 2 \biggl{(} \int_{t_{0}}^{+\infty} \dfrac{\lvert\varphi(t) \rvert^{2}}{t^{2}} \,\mathrm{d}t \biggr{)}^{\!\frac{1}{2}}
\biggl{(}\int_{t_{0}}^{+\infty} \lvert\dot{\varphi}(t) \rvert^{2} \,\mathrm{d}t \biggr{)}^{\!\frac{1}{2}},
\end{align*}
thus proving \eqref{Hardy}. The general case follows from the density of $\mathcal{C}^{\infty}_{\mathrm{c}}(\mathopen{]}t_{0},+\infty\mathclose{[})$ in
$\mathcal{D}^{1,2}_{0}(t_{0},+\infty)$, proved in Proposition~\ref{hilbert}.
\end{proof}

\section{The general strategy}\label{section-4}

In this section we give an outline of the proof of Theorem~\ref{th-main}. Notice that, due to the homogeneity of $U$ and the fact that $\Sigma$ is an open set, it is not restrictive to suppose that
\begin{equation}\label{cone-sigma}
\mathcal{T}(\xi^{+},R,\eta) \subseteq \mathrm{cl}\bigl{(}\mathcal{T}(\xi^{+},R,\eta)\bigr{)} \subseteq \Sigma,
\end{equation}
where $\mathrm{cl}(\mathcal{T}(\xi^{+},R,\eta))$ denotes the closure of $\mathcal{T}(\xi^{+},R,\eta)$, see the definition \eqref{cone}.
This will be implicitly assumed throughout the paper.

\subsection{Changing variables}\label{section-4.1}

The first step in our proof consists in regarding equation \eqref{eq-main} as a perturbation at infinity of the problem
\begin{equation}\label{unperturbed}
\ddot{x} = \nabla U(x), 
\end{equation}
by the introduction of a suitable parameter. This can be done using a scale invariance of the problem; precisely, let $t_{0}\geq 1$ be such that
\begin{equation}\label{cond-t0}
\dfrac{\omega}{2} \, t_{0}^{\frac{2}{\alpha+2}} > R \quad \text{ and } \quad
 \dfrac{2 t_{0}^{\frac{2-\alpha}{2(\alpha+2)}}-1}{2 t_{0}^{\frac{2-\alpha}{2(\alpha+2)}}+1}> \eta,
\end{equation}
and, for every $\varepsilon>0$, $t \geq t_{0}$ and $y \in \mathcal{T}(\xi^{+},\varepsilon^{\frac{3}{2+\alpha}} R,\eta)$, let us set
\begin{equation}\label{def-We}
W_{\varepsilon}(t,y)= \varepsilon^{-\frac{3\alpha}{2+\alpha}} \, W\biggl{(}\dfrac{t-t_{0}}{\varepsilon^{\frac{3}{2}}}, \dfrac{y}{\varepsilon^{\frac{3}{2+\alpha}}}\biggr{)}.
\end{equation}
It is easy to check that, for all $x_{0}\in\mathcal{T}(\xi^{+},R,\eta)$, a function $x \colon \mathopen{[}0,+\infty\mathclose{[} \to \mathbb{R}^{d}$ is a solution of \eqref{eq-main} with $x(0)=x_{0}$ 
if and only if the function $y \colon \mathopen{[}t_{0},+\infty\mathclose{[} \to \mathbb{R}^d$ defined as
\begin{equation*}
y(t) = \varepsilon^{\frac{3}{2+\alpha}} x\biggl{(} \dfrac{t-t_{0}}{\varepsilon^{\frac{3}{2}}}\biggr{)}
\end{equation*}
is a solution of the equation
\begin{equation}\label{eq-y0}
\ddot{y} = \nabla U(y) + \nabla W_{\varepsilon}(t,y), 
\end{equation}
with $y(t_{0})=\varepsilon^{\frac{3}{2+\alpha}} \, x_{0}$. Setting, for $t \geq t_{0}$ and $y \in \mathcal{T}(\xi^{+}, \varepsilon^{\frac{3}{2+\alpha}} R,\eta)$,
\begin{equation}\label{def-Ue}
U_{\varepsilon}(t,y)=  U(y) + W_{\varepsilon}(t,y), 
\end{equation}
equation \eqref{eq-y0} reads as
\begin{equation}\label{eq-y}
\ddot{y} = \nabla U_{\varepsilon}(t,y).
\end{equation}
By the assumptions on $W$, it is readily checked that $U_{\varepsilon}$ can be smoothly extended, when $\varepsilon \to 0^{+}$, to the $-\alpha$-homogeneous potential
\begin{equation}\label{def-U0}
U_{0}(t,y) = U(y)
\end{equation}
(see the beginning of Section~\ref{section-5.1} for more details).  
In this way, we can consider equation \eqref{eq-y} even when $\varepsilon = 0$, provided that we look for solutions $y$ which are bounded away from the origin.

We actually look for solutions to \eqref{eq-y} with a special form. More precisely, denoting by $y_{0}$ the homothetic parabolic solution of the unperturbed problem \eqref{unperturbed} having direction $\xi^{+}$, that is
\begin{equation}\label{def-y0}
y_{0}(t) = \omega \, t^{\frac{2}{\alpha+2}} \xi^{+}, \quad \text{where $\omega=\biggl{(} \dfrac{(\alpha+2)^{2}}{2} \, U(\xi^{+}) \biggr{)}^{\!\frac{1}{\alpha+2}}$,}
\end{equation}
and fixing a function $w \in \mathcal{C}^{2}(\mathopen{[}t_{0},+\infty \mathclose{[})$ satisfying
\begin{equation}\label{def-w}
w(t_{0}) = 1 \quad \text{ and } \quad w(t) = 0, \; \text{for every $t \geq t_{0}+1$,}
\end{equation}
we look for solutions to \eqref{eq-y} of the form
\begin{equation}\label{def-y} 
y(t) = y_{0}(t) + \sigma w(t) + \varphi(t),
\end{equation}
where $\sigma \in \mathbb{R}^{d}$ is a small parameter and $\varphi$ is the new unknown function. 
Setting
\begin{equation}\label{def-ysigma}
y_{\sigma}(t) =  y_{0}(t) + \sigma w(t),
\end{equation}
the equation for $\varphi$ thus becomes the two-parameter equation
\begin{equation*}
\ddot{\varphi} = \nabla U_{\varepsilon}(t,y_{\sigma}(t) + \varphi) - \ddot y_{\sigma}(t),
\end{equation*}
which we will write as
\begin{equation}\label{eq-phi}
\ddot{\varphi} = \nabla K_{\varepsilon,\sigma}(t,\varphi) - h_{\varepsilon,\sigma}(t),
\end{equation}
where 
\begin{equation}\label{def-h}
h_{\varepsilon,\sigma}(t) = \ddot{y}_{\sigma}(t) - \nabla U_{\varepsilon}(t,y_{\sigma}(t))
\end{equation}
and
\begin{equation}\label{def-K}
K_{\varepsilon,\sigma}(t,\varphi) = U_{\varepsilon}(t,y_{\sigma}(t) + \varphi) - U_{\varepsilon}(t,y_{\sigma}(t)) - \bigl{\langle} \nabla U_{\varepsilon}(t,y_{\sigma}(t)),\varphi \bigr{\rangle}.
\end{equation}

The crucial point in our approach is that solutions to \eqref{eq-phi} will be required to belong to the functional space $\mathcal{D}^{1,2}_{0}(t_{0},+\infty)$. Notice that $\varphi(t_{0}) = 0$ implies that $y(t_{0}) = \omega \xi^{+} + \sigma$;
being
\begin{equation}\label{def-xy}
x(t) =  \varepsilon^{-\frac{3}{2+\alpha}}  y\bigl{(} \varepsilon^{\frac{3}{2}}t + t_{0} \bigr{)},
\end{equation}
we thus find
\begin{equation*}
x(0) = \varepsilon^{-\frac{3}{2+\alpha}} \, \bigl{(} \omega \xi^{+} + \sigma \bigr{)}.  
\end{equation*}
As proved in Section~\ref{section-5.3}, by varying $\varepsilon$ and $\sigma$, a set of the form $\mathcal{T}(\xi^{+},R',\eta')$ is thus covered.
Moreover, and more remarkably, due to the fact that $\varphi \in \mathcal{D}^{1,2}_{0}(t_{0},+\infty)$ is a lower order term, for $t \to +\infty$, with respect to $y_{0}$ (since $\alpha<2$),
we will prove (see again Section~\ref{section-5.3}) that the behavior of $x$ at infinity is similar to the one of $y_{0}$: in particular, $x$ is a \textit{parabolic} solution of \eqref{eq-main} asymptotic to the central configuration $\xi^{+}$.

\subsection{A perturbation argument}\label{section-4.2}

In order to find solutions $\varphi \in \mathcal{D}^{1,2}_{0}(t_{0},+\infty)$ to equation \eqref{eq-phi}, we use a perturbative approach: more precisely, being $\varphi \equiv 0$ a solution to \eqref{eq-phi} for $\varepsilon = 0$ and $\sigma = 0$, solutions for $\varepsilon$ and $\sigma$ small enough will be found by an application of the implicit function theorem. 

Taking advantage of the variational structure of \eqref{eq-phi}, we consider the action functional
\begin{equation}\label{def-A}
\mathcal{A}_{\varepsilon,\sigma}(\varphi) = \int_{t_{0}}^{+\infty} \biggl{(} \frac{1}{2} \lvert \dot{\varphi}(t) \rvert^{2} + K_{\varepsilon,\sigma}(t,\varphi(t)) - \bigl{\langle} h_{\varepsilon,\sigma}(t), \varphi(t) \bigr{\rangle} \biggr{)} \,\mathrm{d}t.
\end{equation}
It will be proved (see Proposition~\ref{prop-3.1}) that, when $\varepsilon$ and $\sigma$ are small enough, such a functional is well-defined and of class $\mathcal{C}^{2}$ on a suitable neighborhood of the origin in $\mathcal{D}^{1,2}_{0}(t_{0},+\infty)$ and that the three-variable function 
\begin{equation*}
F \colon (\varepsilon,\sigma,\varphi) \mapsto \mathrm{d} \mathcal{A}_{\varepsilon,\sigma}(\varphi) \in \bigl{(}\mathcal{D}^{1,2}_{0}(t_{0},+\infty)\bigr{)}^{\!*}
\end{equation*}
is continuous and has a continuous differential with respect to $\varphi$ (here, the symbol $\bigl{(}\mathcal{D}^{1,2}_{0}(t_{0},+\infty)\bigr{)}^{\!*}$ denotes the dual space of $\mathcal{D}^{1,2}_{0}(t_{0},+\infty)$). 
By the implicit function theorem, the equation
\begin{equation*}
\mathrm{d} \mathcal{A}_{\varepsilon,\sigma}(\varphi) = 0
\end{equation*}
can thus be solved with respect to $\varphi$, for $\varepsilon$ and $\sigma$ small enough, whenever 
\begin{equation*}
\mathrm{D}_{\varphi} F(0,0,0) = \mathrm{d}^{2} \mathcal{A}_{0,0}(0)
\end{equation*}
is invertible. We will see that this is actually the case (see \eqref{diff-for} and Proposition~\ref{prop-3.2}), thus providing critical points $\varphi = \varphi(\varepsilon,\sigma)$ of $\mathcal{A}_{\varepsilon,\sigma}(\varphi)$ and, eventually, solutions to equation \eqref{eq-phi} belonging to the space $\mathcal{D}^{1,2}_{0}(t_{0},+\infty)$.

\section{Proof of Theorem~\ref{th-main}}\label{section-5}

In this section we give the proof of Theorem~\ref{th-main}. More precisely, in Section~\ref{section-5.1} we establish some preliminary estimates, while in Section~\ref{section-5.2} we illustrate the perturbation argument. The proof will be finally described in Section~\ref{section-5.3}.

\begin{remark}\label{rem-moregeneral}
By a careful inspection of the proof one could check that it is sufficient to assume that $W$ is only continuous with respect to time.
\end{remark}

\subsection{Preliminary estimates}\label{section-5.1}

Before starting with the proof, we first observe that, from assumption \eqref{hp-main}, we can fix a constant $C > 0$ such that
\begin{align}
&\lvert W(t,x) \rvert \leq \frac{C}{\lvert x \rvert^{\beta}},&
&\text{for every $(t,x)\in\mathopen{[}0,+\infty\mathclose{[}\times \mathcal{T}(\xi^+,R,\eta)$,}& \label{W1}\\
&\lvert \nabla W(t,x) \rvert \leq \frac{C}{\lvert x \rvert^{\beta+1}},&
&\text{for every $(t,x)\in\mathopen{[}0,+\infty\mathclose{[}\times \mathcal{T}(\xi^+,R,\eta)$,}& \label{W2}\\
&\lvert \nabla^{2} W(t,x) \rvert \leq \frac{C}{\lvert x \rvert^{\beta+2}},&
&\text{for every $(t,x)\in\mathopen{[}0,+\infty\mathclose{[}\times \mathcal{T}(\xi^+,R,\eta)$.}& \label{W3}\\
\end{align}
As a consequence, for every $\varepsilon \in \mathopen{]}0,1\mathclose{[}$, the potential $W_{\varepsilon}$ defined in \eqref{def-We} satisfies
\begin{align}
&\lvert W_{\varepsilon}(t,y) \rvert \leq \frac{C \varepsilon^{\frac{3(\beta-\alpha)}{2+\alpha}}}{\lvert y \rvert^{\beta}},&
& \text{for every $(t,y) \in \mathopen{[}t_{0},+\infty\mathclose{[}\times \mathcal{T}(\xi^{+},\varepsilon^{\frac{3}{2+\alpha}} R,\eta)$,}& \label{We1} \\
&\lvert \nabla W_{\varepsilon}(t,y) \rvert \leq \frac{C \varepsilon^{\frac{3(\beta-\alpha)}{2+\alpha}}}{\lvert y \rvert^{\beta+1}},&
& \text{for every $(t,y) \in \mathopen{[}t_{0},+\infty\mathclose{[}\times \mathcal{T}(\xi^{+},\varepsilon^{\frac{3}{2+\alpha}} R,\eta)$,}& \label{We2} \\
&\lvert \nabla^{2} W_{\varepsilon}(t,y) \rvert \leq \frac{C \varepsilon^{\frac{3(\beta-\alpha)}{2+\alpha}}}{\lvert y \rvert^{\beta+2}},&
& \text{for every $(t,y) \in \mathopen{[}t_{0},+\infty\mathclose{[}\times \mathcal{T}(\xi^{+},\varepsilon^{\frac{3}{2+\alpha}} R,\eta)$.}& \label{We3} \\
\end{align}

By the $-\alpha$-homogeneity of $U$, the fact that $U(x) = |x|^{-\alpha} \, U(x/|x|)$ for all $x\in\mathrm{cl}\bigl{(}\mathcal{T}(\xi^{+},R,\eta)\bigr{)} \subseteq \Sigma$, together with similar equalities for $\nabla U$ and $\nabla^{2} U$, leads to the following inequalities
\begin{align}
&\lvert U(x) \rvert \leq \dfrac{C'}{|x|^{\alpha}},&
&\text{for every $x\in\mathcal{T}(\xi^+,R,\eta)$,}& \label{cond-U0}\\
&\lvert \nabla U(x)\rvert \leq \dfrac{C'}{|x|^{\alpha+1}},&
&\text{for every $x\in\mathcal{T}(\xi^+,R,\eta)$,}& \label{cond-U1}\\
&\lvert \nabla^{2} U(x) \rvert \leq \dfrac{C'}{|x|^{\alpha+2}},&
&\text{for every $x\in\mathcal{T}(\xi^+,R,\eta)$,}& \label{cond-U2}
\end{align}
where $C'>0$ is a suitable constant. Recalling the definitions of $U_\varepsilon$ in \eqref{def-Ue} and of $U_{0}$ in \eqref{def-U0}, from conditions \eqref{cond-U0}, \eqref{cond-U1}, \eqref{cond-U2} and the fact that $\alpha<\beta$, we find $C'' > 0$ such that, for every $\varepsilon \in \mathopen{[}0,1\mathclose{[}$,
\begin{align}
&\lvert U_{\varepsilon}(t,y) \rvert \leq \frac{C''}{\lvert y \rvert^{\alpha}},&
& \text{for every $(t,y) \in \mathopen{[}t_{0},+\infty\mathclose{[}\times \mathcal{T}(\xi^+,R,\eta)$,}& \label{Ue1} \\
&\lvert \nabla U_{\varepsilon}(t,y) \rvert \leq \frac{C''}{\lvert y \rvert^{\alpha+1}},&
& \text{for every $(t,y) \in \mathopen{[}t_{0},+\infty\mathclose{[}\times \mathcal{T}(\xi^+,R,\eta)$,}& \label{Ue2} \\
&\lvert \nabla^{2} U_{\varepsilon}(t,y) \rvert \leq \frac{C''}{\lvert y \rvert^{\alpha+2}},&
& \text{for every $(t,y) \in \mathopen{[}t_{0},+\infty\mathclose{[}\times \mathcal{T}(\xi^+,R,\eta)$.}& \label{Ue3}
\end{align}
Let us notice that from \eqref{We1}, \eqref{We2}, \eqref{We3} and the fact that $\alpha<\beta$ it easily follows that 
\begin{equation}\label{convergenza}
U_{\varepsilon}(t,y) \to U(y), \quad \text{as $\varepsilon \to 0^{+}$},
\end{equation}
where the above convergence is meant in the sense that 
$U_{\varepsilon}(t,y) \to U(y)$ and $\nabla^{i} U_{\varepsilon}(t,y) \to \nabla^{i} U(y)$, for $i=1,2$, uniformly in $\mathopen{[}t_{0},+\infty\mathclose{[}\times \mathcal{T}(\xi^+,R,\eta)$. 

We are now ready to start with the proof. As a first step, we are going to fix some constants; precisely let us set
\begin{equation*}
r = \dfrac{\omega}{\displaystyle 4 \max_{t \in \mathopen{[}t_{0},t_{0}+1\mathclose{]}} \lvert w(t) \rvert} 
\quad \text{ and } \quad
\rho = \dfrac{\omega}{4},
\end{equation*}
where the function $w$ is as in \eqref{def-w}. 
Accordingly, we define the sets
\begin{equation*}
B_{r} = \bigl{\{} \sigma \in \mathbb{R}^{d} \colon \lvert \sigma \rvert < r \bigr{\}}
\quad \text{ and } \quad
\Omega_{\rho} = \Bigl{\{} \varphi \in \mathcal{D}^{1,2}_{0}(t_{0},+\infty) \colon \Vert \varphi \Vert < \rho \Bigr{\}}.
\end{equation*}
Then, recalling the definition of $y_\sigma$ given in \eqref{def-ysigma}, we have the following preliminary result.

\begin{lemma}\label{stima}
For every $\sigma \in B_{r}$ and for every  $\varphi \in \Omega_{\rho}$, the following estimates hold true:
\begin{align}
&\lvert y_{\sigma}(t) + \varphi(t) \rvert \geq \dfrac{\omega}{2} \, t^{\frac{2}{\alpha+2}}, \quad \text{for every $t \geq t_{0}$,}
\label{y+phi}
\\
&\Big{\langle} \dfrac{y_{\sigma}(t) + \varphi(t)}{\lvert y_{\sigma}(t) + \varphi(t) \rvert}, \xi^{+} \Big{\rangle} > \dfrac{2 t^{\frac{2-\alpha}{2(\alpha+2)}}-1}{2 t^{\frac{2-\alpha}{2(\alpha+2)}}+1}, \quad \text{for every $t \geq t_{0}$.}
\label{y+phi++}
\end{align}
\end{lemma}

\begin{proof}
Using \eqref{Hardy-sqrt} we have
\begin{align*}
\lvert y_{\sigma}(t) + \varphi(t) \rvert &\geq \lvert y_{0}(t) \rvert - \lvert \sigma w(t) \rvert - \lvert \varphi(t) \rvert \\
& \geq t^{\frac{2}{\alpha+2}} \biggl{(} \omega - \lvert \sigma \rvert \max_{t \in \mathopen{[}t_{0},t_{0}+1\mathclose{]}}\lvert w(t) \rvert t^{-\frac{2}{\alpha+2}} - \Vert \varphi \Vert t^{\frac{\alpha-2}{2(\alpha+2)}} \biggr{)} \\
& \geq t^{\frac{2}{\alpha+2}} \biggl{(}  \omega - r \max_{t \in \mathopen{[}t_{0},t_{0}+1\mathclose{]}}\lvert w(t) \rvert  - \rho \biggr{)} = \dfrac{\omega}{2} \, t^{\frac{2}{\alpha+2}},
\end{align*}
for every $t \geq t_{0}$. Therefore, \eqref{y+phi} is proved. As for \eqref{y+phi++}, we have
\begin{align*}
\Big{\langle} \dfrac{y_{\sigma}(t) + \varphi(t)}{\lvert y_{\sigma}(t) + \varphi(t) \rvert}, \xi^{+} \Big{\rangle}
&= \dfrac{ \omega t^{\frac{2}{\alpha+2}} + \sigma \langle w(t),\xi^{+} \rangle + \langle \varphi(t),\xi^{+} \rangle }{\lvert y_{0}(t) + \sigma w(t) + \varphi(t) \rvert}
\\
&\geq \dfrac{ \omega t^{\frac{2}{\alpha+2}} - \lvert \sigma \rvert \displaystyle \max_{t \in \mathopen{[}t_{0},t_{0}+1\mathclose{]}}\lvert w(t) \rvert - \|\varphi\| \sqrt{t}}{ \omega t^{\frac{2}{\alpha+2}} + \lvert \sigma \rvert \displaystyle \max_{t \in \mathopen{[}t_{0},t_{0}+1\mathclose{]}}\lvert w(t) \rvert + \|\varphi\| \sqrt{t}}
\\
&\geq \dfrac{ \omega t^{\frac{2-\alpha}{2(\alpha+2)}} - \lvert \sigma \rvert \displaystyle \max_{t \in \mathopen{[}t_{0},t_{0}+1\mathclose{]}}\lvert w(t) \rvert - \|\varphi\|}{ \omega t^{\frac{2-\alpha}{2(\alpha+2)}} + \lvert \sigma \rvert \displaystyle \max_{t \in \mathopen{[}t_{0},t_{0}+1\mathclose{]}}\lvert w(t) \rvert + \|\varphi\|}
\\
&> \dfrac{2 t^{\frac{2-\alpha}{2(\alpha+2)}} -1}{2 t^{\frac{2-\alpha}{2(\alpha+2)}} +1},
\end{align*}
for every $t \geq t_{0}$, where the last inequality follows from the fact that
\begin{equation*}
\lvert \sigma \rvert \max_{t \in \mathopen{[}t_{0},t_{0}+1\mathclose{]}}\lvert w(t) \rvert  + \|\varphi\| 
<  r \max_{t \in \mathopen{[}t_{0},t_{0}+1\mathclose{]}}\lvert w(t) \rvert + \rho = \dfrac{\omega}{2}.
\end{equation*}
Therefore, \eqref{y+phi++} is proved.
\end{proof}

Notice in particular that, by assumption \eqref{cond-t0}, for every $\sigma \in B_{r}$ and for every $\varphi \in \Omega_{\rho}$, it holds  
\begin{align*}
&\lvert y_{\sigma}(t) + \varphi(t) \rvert > R, \quad \text{for every $t \geq t_{0}$,}
\\
&\Big{\langle} \dfrac{y_{\sigma}(t) + \varphi(t)}{\lvert y_{\sigma}(t) + \varphi(t) \rvert}, \xi^{+} \Big{\rangle} > \eta, \quad \text{for every $t \geq t_{0}$,}
\end{align*}
so that the functions $h_{\varepsilon,\sigma}$ and $K_{\varepsilon,\sigma}$ (see \eqref{def-h} and \eqref{def-K}) are well-defined for 
$\varepsilon \in \mathopen{[}0,1\mathclose{[}$.
In particular, due to \eqref{convergenza}, the definitions are meaningful also when $\varepsilon = 0$.

Our next two lemmas give some estimates for $h_{\varepsilon,\sigma}$ and $K_{\varepsilon,\sigma}$.

\begin{lemma}\label{lem-h}
There exists $C_{h} > 0$ such that, for every $\varepsilon \in \mathopen{[}0,1\mathclose{[}$ and $\sigma \in B_{r}$,
\begin{equation}\label{eq-h-C}
\lvert h_{\varepsilon,\sigma}(t) \rvert \leq C_{h} \, t^{-\frac{2(\beta+1)}{\alpha+2}}, \quad \text{for every $t \geq t_{0}$}.
\end{equation}
\end{lemma}

\begin{proof}
In order to prove \eqref{eq-h-C}, we first suppose that $t \in \mathopen{[}t_{0},t_{0}+1\mathclose{]}$; notice that in this case we just need to show that $\max_{t \in \mathopen{[}t_{0},t_{0}+1\mathclose{]}} \lvert h_{\varepsilon,\sigma}(t) \rvert$ is bounded, independently of $\varepsilon$ and $\sigma$.
This is easily checked: indeed, on one hand by construction $y_{\sigma} \to y_{0}$ in $\mathcal{C}^{2}(\mathopen{[}t_{0},t_{0}+1\mathclose{]})$ as $\sigma\to 0$; on the other hand, since $\lvert y_{\sigma}(t) \rvert > R$ we have that $\nabla U_{\varepsilon}(t,y_{\sigma}(t))$ is bounded, uniformly in $\varepsilon$ and $\sigma$, by \eqref{Ue2}.

We now suppose that $t \geq t_{0}+1$; in this case, recalling that $y_{\sigma}(t) = y_{0}(t)$, we find 
\begin{equation*}
h_{\varepsilon,\sigma}(t) = \ddot{y}_{0}(t) - \nabla U_{\varepsilon}(t,y_{0}(t)) = -\nabla W_{\varepsilon}(t,y_{0}(t)),
\end{equation*}
for every $t \geq t_{0}+1$, and the conclusion follows by \eqref{We2}.
\end{proof}

\begin{lemma}\label{lem-K}
There exists $C_{K} > 0$ such that, for every $\varepsilon\in\mathopen{[}0,1\mathclose{[}$, $\sigma\in B_{r}$ and $\varphi\in\Omega_{\rho}$, the following inequalities hold true:
\begin{align}
&\lvert K_{\varepsilon,\sigma}(t,\varphi(t)) \rvert \leq \frac{C_{K} \lvert \varphi(t) \rvert^{2}}{t^{2}}, &\text{for every $t \geq t_{0}$,}& \label{K1} \\
&\lvert \nabla K_{\varepsilon,\sigma}(t,\varphi(t)) \rvert \leq \frac{C_{K} \lvert \varphi(t) \rvert}{t^{2}}, &\text{for every $t \geq t_{0}$,}& \label{K2} \\
&\lvert \nabla^{2} K_{\varepsilon,\sigma}(t,\varphi(t)) \rvert \leq \frac{C_{K}}{t^{2}}, &\text{for every $t \geq t_{0}$.}& \label{K3}
\end{align}
\end{lemma}

\begin{proof}
We have
\begin{equation*}
\nabla K_{\varepsilon,\sigma}(t,\varphi(t)) = \nabla U_{\varepsilon} (t,y_{\sigma}(t)+\varphi(t)) - \nabla U_{\varepsilon}(t,y_{\sigma}(t))
\end{equation*}
and
\begin{equation*}
\nabla^{2} K_{\varepsilon,\sigma}(t,\varphi(t)) = \nabla^{2} U_{\varepsilon} (t,y_{\sigma}(t)+\varphi(t)).
\end{equation*}
for every $t\geq t_{0}$. Moreover, from \eqref{y+phi}, for every $\sigma\in B_{r}$ and $\varphi\in\Omega_{\rho}$ we immediately find that
\begin{equation}\label{stima-cubo}
\frac{1}{\lvert y_{\sigma}(t) + \varphi(t) \rvert^{\alpha+2}} \leq \dfrac{2^{\alpha+2}}{\omega^{\alpha+2}} \dfrac{1}{t^{2}}, \quad \text{for every $t \geq t_{0}$.}
\end{equation}
Combining \eqref{Ue3} with \eqref{stima-cubo}, \eqref{K3} plainly follows.
To prove \eqref{K2}, we notice that $\nabla K_{\varepsilon,\sigma}(t,0) \equiv 0$ so as to write
\begin{equation*}
\nabla K_{\varepsilon,\sigma}(t,\varphi(t)) = \int_{0}^{1} \frac{\mathrm{d}}{\mathrm{d}s} \nabla K_{\varepsilon,\sigma}(t,s\varphi(t))\,\mathrm{d}s
= \int_{0}^{1} \nabla^{2} K_{\varepsilon,\sigma}(t,s\varphi(t))\varphi(t)\,\mathrm{d}s,
\end{equation*}
for every $t \geq t_{0}$. Hence, \eqref{K2} directly follows from \eqref{K3} (with $s\varphi$ in place of $\varphi$). In an analogous way (using $K_{\varepsilon,\sigma}(t,0) \equiv 0$) we obtain \eqref{K1} from \eqref{K2}.
\end{proof}

\subsection{The implicit function argument}\label{section-5.2}

In this section, we provide the details of the implicit function argument. To this end, we start by recalling the definition of the action functional $\mathcal{A}_{\varepsilon,\sigma}$ given in \eqref{def-A}; notice that we now assume $\sigma \in B_{r}$ and 
$\varphi \in \Omega_{\rho}$ so that all the integrands are well-defined. 

\begin{proposition}\label{prop-3.1}
For every $\varepsilon\in\mathopen{[}0,1\mathclose{[}$ and $\sigma\in B_{r}$, the action functional $\mathcal{A}_{\varepsilon,\sigma}$ is of class $\mathcal{C}^{2}$ on the open set $\Omega_{\rho} \subseteq \mathcal{D}^{1,2}_{0}(t_{0},+\infty)$, with
\begin{equation*}
\mathrm{d} \mathcal{A}_{\varepsilon,\sigma}(\varphi)[\psi] =  \int_{t_{0}}^{+\infty} \Bigl{(} \bigl{\langle} \dot{\varphi}(t),\dot{\psi}(t) \bigr{\rangle} +  \bigl{\langle} \nabla K_{\varepsilon,\sigma}(t,\varphi(t)),\psi(t) \bigr{\rangle} + \bigl{\langle} h_{\varepsilon,\sigma}(t),\psi(t) \bigr{\rangle} \Bigr{)}\,\mathrm{d}t
\end{equation*}
and
\begin{equation*}
\mathrm{d}^{2} \mathcal{A}_{\varepsilon,\sigma}(\varphi)[\psi,\zeta]  
= \int_{t_{0}}^{+\infty} \Bigl{(} \bigl{\langle}  \dot{\psi}(t),\dot{\zeta}(t) \bigr{\rangle} +  \bigl{\langle} \nabla^{2} K_{\varepsilon,\sigma}(t,\varphi(t))\psi(t),\zeta(t) \bigr{\rangle} \Bigr{)}\,\mathrm{d}t,
\end{equation*}
for every $\psi, \zeta \in \mathcal{D}^{1,2}_{0}(t_{0},+\infty)$. Moreover, the three-variable function 
\begin{equation*}
F \colon \mathopen{[}0,1\mathclose{[} \times B_{r} \times \Omega_{\rho} \to \bigl{(}\mathcal{D}^{1,2}_{0}(t_{0},+\infty)\bigr{)}^{\!*}, \qquad
F(\varepsilon,\sigma,\varphi) = \mathrm{d}\mathcal{A}_{\varepsilon,\sigma}(\varphi),
\end{equation*}
is continuous and its differential $\mathrm{D}_{\varphi} F$ is continuous. 
\end{proposition}

\begin{proof}
We split the action functional as
\begin{equation}\label{split}
\mathcal{A}_{\varepsilon,\sigma}(\varphi) = \mathcal{A}^{1}(\varphi) + \mathcal{A}^{2}_{\varepsilon,\sigma}(\varphi) + \mathcal{A}^{3}_{\varepsilon,\sigma}(\varphi),
\end{equation}
where $\mathcal{A}^{1}(\varphi) = \tfrac{1}{2} \Vert \varphi \Vert^{2}$,
\begin{equation*}
\mathcal{A}^{2}_{\varepsilon,\sigma}(\varphi) = \int_{t_{0}}^{+\infty} K_{\varepsilon,\sigma}(t,\varphi(t)) \,\mathrm{d}t, \quad 
\mathcal{A}^{3}_{\varepsilon,\sigma}(\varphi) = \int_{t_{0}}^{+\infty} \bigl{\langle} h_{\varepsilon,\sigma}(t), \varphi(t) \bigr{\rangle}  \,\mathrm{d}t,
\end{equation*}
and we investigate each term separately.

As for $\mathcal{A}^{1}$, it is well known that it is a smooth functional, with
\begin{equation*}
\mathrm{d} \mathcal{A}^{1}(\varphi)[\psi] = \int_{t_{0}}^{+\infty} \bigl{\langle} \dot{\varphi}(t),\dot{\psi}(t) \bigr{\rangle} \,\mathrm{d}t \quad \text{ and }  
\quad \mathrm{d}^{2} \mathcal{A}^{1}(\varphi)[\psi,\zeta] = \int_{t_{0}}^{+\infty} \bigl{\langle} \dot{\psi}(t),\dot{\zeta}(t) \bigr{\rangle} \,\mathrm{d}t.
\end{equation*}
The continuity with respect to the parameters $\varepsilon$ and $\sigma$ is obvious, since they do not appear in the above expressions.

We now focus our attention on the (linear) term $\mathcal{A}^{3}_{\varepsilon,\sigma}$. Using \eqref{Hardy-sqrt} and \eqref{eq-h-C}, we deduce
\begin{equation*}
\biggl{\lvert}\int_{t_{0}}^{+\infty}\bigl{\langle} h_{\varepsilon,\sigma}(t), \varphi(t) \bigr{\rangle} \,\mathrm{d}t \biggr{\rvert} \leq C_{h} \Vert \varphi \Vert \int_{t_{0}}^{+\infty} t^{-\frac{4\beta-\alpha+2}{2(\alpha+2)}} \,\mathrm{d}t < +\infty,
\end{equation*}
since $4\beta-3\alpha>2$, so that $\mathcal{A}^{3}_{\varepsilon,\sigma}$ is well-defined and continuous. Therefore, 
\begin{equation*}
\mathrm{d} \mathcal{A}^{3}_{\varepsilon,\sigma}(\varphi)[\psi] = \int_{t_{0}}^{+\infty} \bigl{\langle} h_{\varepsilon,\sigma}(t), \psi(t) \bigr{\rangle} \,\mathrm{d}t
\quad \text{ and } \quad \mathrm{d}^{2} \mathcal{A}^{3}_{\varepsilon,\sigma}(\varphi)[\psi,\zeta] = 0.
\end{equation*}
The only thing to check is the continuity of the map $(\varepsilon,\sigma,\varphi) \mapsto \mathrm{d} \mathcal{A}^{3}_{\varepsilon,\sigma}(\varphi)$. Precisely, since $\mathrm{d} \mathcal{A}^{3}_{\varepsilon,\sigma}(\varphi)$ does not depend on $\varphi$, we need to verify that
\begin{equation}\label{h-cont0}
(\varepsilon_{n},\sigma_{n}) \to (\varepsilon,\sigma) \quad \Rightarrow \quad \sup_{\Vert \psi \Vert \leq 1} 
\biggl{\lvert} \int_{t_{0}}^{+\infty} \bigl{\langle} h_{\varepsilon_{n},\sigma_{n}}(t) - h_{\varepsilon,\sigma}(t), \psi(t) \bigr{\rangle} \,\mathrm{d}t \biggr{\rvert} \to 0.
\end{equation}
To prove this, we first use the Cauchy--Schwartz inequality together with \eqref{Hardy} to get 
\begin{equation}\label{h-cont}
\begin{aligned}
&\sup_{\Vert \psi \Vert \leq 1} 
\biggl{\lvert} \int_{t_{0}}^{+\infty} \bigl{\langle} h_{\varepsilon_{n},\sigma_{n}}(t) - h_{\varepsilon,\sigma}(t), \psi(t) \bigr{\rangle} \,\mathrm{d}t \biggr{\rvert} \leq
\\& \leq \sup_{\Vert \psi \Vert \leq 1} \int_{t_{0}}^{+\infty} t \lvert h_{\varepsilon_{n},\sigma_{n}}(t) - h_{\varepsilon,\sigma}(t) \rvert
 \dfrac{\lvert \psi(t) \rvert}{t} \,\mathrm{d}t 
\\& \leq \sup_{\Vert \psi \Vert \leq 1} 
\biggl{(} \int_{t_{0}}^{+\infty} \dfrac{\lvert \psi(t) \rvert^{2}}{t^{2}}\,\mathrm{d}t \biggr{)}^{\!\frac{1}{2}}  
\biggl{(} \int_{t_{0}}^{+\infty} t^{2} \lvert h_{\varepsilon_{n},\sigma_{n}}(t) - h_{\varepsilon,\sigma}(t)\rvert^{2}\,\mathrm{d}t \biggr{)}^{\!\frac{1}{2}}  
\\&\leq 2 \, \biggl{(} \int_{t_{0}}^{+\infty} t^{2} \lvert h_{\varepsilon_{n},\sigma_{n}}(t) - h_{\varepsilon,\sigma}(t)\rvert^{2}\,\mathrm{d}t \biggr{)}^{\!\frac{1}{2}}.
\end{aligned}
\end{equation}
We then conclude, using the dominated convergence theorem: indeed, it is easily checked that 
$h_{\varepsilon_{n},\sigma_{n}}(t) \to h_{\varepsilon,\sigma}(t)$ pointwise
and the integrand is bounded by the integrable function $4 \, C_{h}^{2} \, t^{-\frac{2(2\beta-\alpha)}{\alpha+2}}$ (recall that $4\beta-3\alpha>2$), by \eqref{eq-h-C}.

We finally deal with the (nonlinear) term $\mathcal{A}^{2}_{\varepsilon,\sigma}$. Incidentally, observe that the integral is well-defined, as it can be seen by combining \eqref{K1} with \eqref{Hardy}.

Let us consider the linear form 
\begin{equation}\label{da2}
L_{\varepsilon,\sigma,\varphi}[\psi] = \int_{t_{0}}^{+\infty} \bigl{\langle} \nabla K_{\varepsilon,\sigma}(t,\varphi(t)),\psi(t) \bigr{\rangle} \,\mathrm{d}t, \qquad \psi \in \mathcal{D}^{1,2}_{0}(t_{0},+\infty),
\end{equation}
and the bilinear form 
\begin{equation}\label{d2a2}
B_{\varepsilon,\sigma,\varphi}[\psi,\zeta] = \int_{t_{0}}^{+\infty} \bigl{\langle} \nabla^{2} K_{\varepsilon,\sigma}(t,\varphi(t))\psi(t),\zeta(t) \bigr{\rangle} \,\mathrm{d}t, \qquad \psi,\zeta \in \mathcal{D}^{1,2}_{0}(t_{0},+\infty).
\end{equation}
Notice that $L_{\varepsilon,\sigma,\varphi}$ and $B_{\varepsilon,\sigma,\varphi}$ are actually well-defined and continuous: indeed, using \eqref{K2} together with the Cauchy--Schwartz inequality and \eqref{Hardy}, we obtain
\begin{align*}
\lvert L_{\varepsilon,\sigma,\varphi}[\psi] \rvert & \leq C_{K} \int_{t_{0}}^{+\infty} \dfrac{\lvert \varphi(t) \rvert \lvert \psi(t) \rvert}{t^{2}}\,\mathrm{d}t \\
& \leq 
C_{K} \biggl{(} \int_{t_{0}}^{+\infty} \frac{\lvert \varphi(t) \rvert^{2}}{t^{2}}\,\mathrm{d}t\biggr{)}^{\!\frac{1}{2}} \biggl{(} \int_{t_{0}}^{+\infty} \frac{\lvert \psi(t) \rvert^{2}}{t^{2}}\,\mathrm{d}t\biggr{)}^{\!\frac{1}{2}}  \leq 4 C_{K} \Vert \varphi \Vert \Vert \psi \Vert.
\end{align*}
An analogous argument works for $B_{\varepsilon,\sigma,\varphi}$ (using \eqref{K3}).

We now prove that the functions $(\varepsilon,\sigma,\varphi) \mapsto L_{\varepsilon,\sigma,\varphi}$ and 
$(\varepsilon,\sigma,\varphi) \mapsto B_{\varepsilon,\sigma,\varphi}$
are continuous as functions with values in the space of linear and bilinear forms on $\mathcal{D}^{1,2}_{0}(t_{0},+\infty)$, respectively, that is,
\begin{equation}\label{L-cont}
(\varepsilon_{n},\sigma_{n},\varphi_{n}) \to (\varepsilon,\sigma,\varphi) \quad \Rightarrow \quad \sup_{\Vert \psi \Vert \leq 1} \lvert 
L_{\varepsilon_{n},\sigma_{n},\varphi_{n}}[\psi] - L_{\varepsilon,\sigma,\varphi}[\psi] \rvert \to 0 
\end{equation}
and
\begin{equation}\label{B-cont}
(\varepsilon_{n},\sigma_{n},\varphi_{n}) \to (\varepsilon,\sigma,\varphi) \quad \Rightarrow \quad 
\sup_{\substack{\Vert \psi \Vert \leq 1\\  \Vert \zeta \Vert \leq 1}} \lvert B_{\varepsilon_{n},\sigma_{n},\varphi_{n}}[\psi,\zeta] - 
B_{\varepsilon,\sigma,\varphi}[\psi,\zeta] \rvert \to 0. 
\end{equation}

We start with \eqref{L-cont}. Arguing as in \eqref{h-cont}, we find
\begin{align*}
& \sup_{\Vert \psi \Vert \leq 1} \lvert
L_{\varepsilon_{n},\sigma_{n},\varphi_{n}}[\psi] - L_{\varepsilon,\sigma,\varphi}[\psi] \rvert \leq \\
& \leq 2 \, \biggl{(} \int_{t_{0}}^{+\infty} t^{2} \lvert \nabla K_{\varepsilon_{n},\sigma_{n}}(t,\varphi_{n}(t)) - \nabla K_{\varepsilon,\sigma}(t,\varphi(t))\rvert^{2}\,\mathrm{d}t \biggr{)}^{\!\frac{1}{2}}.
\end{align*}
We are going to show that the integral goes to zero, using the dominated convergence theorem. To this end, we first observe that the integrand goes to zero pointwise, since $\varphi_{n} \to \varphi$ in $\mathcal{D}^{1,2}_{0}(t_{0},+\infty)$ implies uniform convergence on compact sets (recall the inequality \eqref{Hardy-sqrt}). To prove that the integrand is $L^{1}$-bounded, we use \eqref{K2} and elementary inequalities so as to obtain
\begin{align*}
& t^{2} \lvert \nabla K_{\varepsilon_{n},\sigma_{n}}(t,\varphi_{n}(t)) - \nabla K_{\varepsilon,\sigma}(t,\varphi(t))\rvert^{2} \leq 2 C_{K} 
\biggl{(} \frac{\lvert\varphi_{n}(t)\rvert^{2}}{t^{2}} + \frac{\lvert\varphi(t)\rvert^{2}}{t^{2}} \biggr{)} \\
& \leq 4 C_{K} 
\biggl{(} \frac{\lvert\varphi(t)\rvert^{2}}{t^{2}} + \frac{\lvert\varphi_{n}(t)-\varphi(t)\rvert^{2}}{t^{2}} \biggr{)},
\end{align*}
for every $t \geq t_{0}$. By Hardy inequality \eqref{Hardy}, the first term on the right-hand side is in $L^{1}$; on the other hand, again by Hardy inequality, the second term goes to zero in $L^{1}$ and thus, up to a subsequence, is $L^{1}$-dominated. Hence, the dominated convergence theorem applies along a subsequence
and a standard argument yields the conclusion for the original sequence. 

We now prove \eqref{B-cont}; this will require a more careful analysis. Recalling the definition of $K_{\varepsilon,\sigma}$, we are going to show that
\begin{equation}\label{primopezzo}
\sup_{\substack{\Vert \psi \Vert \leq 1\\  \Vert \zeta \Vert \leq 1}} \biggl{\lvert} \int_{t_{0}}^{+\infty} \bigl{\langle}  \bigl{(}\nabla^{2} U(y_{\sigma_{n}}(t)+\varphi_{n}(t)) - \nabla^{2} U(y_{\sigma}(t)+\varphi(t))\bigr{)}\psi(t),\zeta(t)\bigr{\rangle} \,\mathrm{d}t \biggr{\rvert}\to 0
\end{equation}
and that
\begin{equation}\label{secondopezzo}
\sup_{\substack{\Vert \psi \Vert \leq 1\\  \Vert \zeta \Vert \leq 1}}  \biggl{\lvert} \int_{t_{0}}^{+\infty} \bigl{\langle}
\bigl{(}\nabla^{2} W_{\varepsilon_{n}}(t,y_{\sigma_{n}}(t)+\varphi_{n}(t)) - \nabla^{2} W_{\varepsilon}(t,y_{\sigma}(t)+\varphi(t))\bigr{)}\psi(t),\zeta(t)\bigr{\rangle} \,\mathrm{d}t \biggr{\rvert} \to 0,
\end{equation}
as $(\varepsilon_n,\sigma_n,\varphi_n) \to (\varepsilon,\sigma,\varphi)$.

We first deal with \eqref{primopezzo}. Preliminarily, we observe that, denoting by $z_{n,\lambda}(t)$ a generic point along the segment joining $y_{\sigma}(t) + \varphi(t)$ with $y_{\sigma_{n}}(t) + \varphi_{n}(t)$, that is, for $\lambda \in \mathopen{[}0,1\mathclose{]}$ and $t \geq t_{0}$,
\begin{equation*}
z_{n,\lambda}(t) = y_{\sigma}(t) + \varphi(t) + \lambda \bigl{(} y_{\sigma_{n}}(t) - y_{\sigma}(t) + \varphi_{n}(t) - \varphi(t) \bigr{)},
\end{equation*}
the estimate
\begin{equation}\label{cruciale}
\lvert z_{n,\lambda}(t) \rvert \geq \dfrac{\omega}{4} \, t^{\frac{2}{\alpha+2}}, \quad \text{for every $t \geq t_{0}$},
\end{equation}
holds true, when $n$ is large enough. Indeed, using \eqref{y+phi} and \eqref{Hardy-sqrt} (for $\varphi_{n} - \varphi$), we obtain
\begin{align*}
\lvert z_{n,\lambda}(t) \rvert & \geq \lvert y_{\sigma}(t) + \varphi(t) \rvert - \lvert y_{\sigma_{n}}(t) - y_{\sigma}(t) \rvert - \lvert   \varphi_{n}(t) - \varphi(t)\rvert \\
& \geq t^{\frac{2}{\alpha+2}} \biggl{(} \frac{\omega}{2} - \lvert \sigma_{n} - \sigma \rvert \max_{t \in \mathopen{[}1,2\mathclose{]}} \lvert w(t) \rvert t^{-\frac{2}{\alpha+2}} - \Vert \varphi_{n} - \varphi \Vert t^{-\frac{\alpha-2}{2(\alpha+2)}} \biggr{)} \\
& \geq t^{\frac{2}{\alpha+2}} \biggl{(} \frac{\omega}{2} - \lvert \sigma_{n} - \sigma \rvert \max_{t \in \mathopen{[}1,2\mathclose{]}} \lvert w(t) \rvert - \Vert \varphi_{n} - \varphi \Vert \biggr{)},
\end{align*}
for every $t \geq t_{0}$, whence the conclusion when $n$ is large enough. In particular, we have that $z_{n,\lambda}(t) \in \mathcal{T}(\xi^+,R/2,\eta) \subseteq \Sigma$ for every $t \geq t_{0}$, when $n$ is large enough. Therefore, we can use the mean value theorem to obtain
\begin{align*}
& \lvert \nabla^{2} U(y_{\sigma_{n}}(t)+\varphi_{n}(t)) - \nabla^{2} U(y_{\sigma}(t)+\varphi(t)) \rvert \leq \\
& \leq \tilde{C} \sup_{\substack{\lambda \in \mathopen{[}0,1\mathclose{]} \\ i,j,k \in \{1,\ldots,d\}}}\lvert \partial^{3}_{ijk} U(z_{n,\lambda}(t)) \rvert \bigl{(}\lvert y_{\sigma_{n}}(t) - y_{\sigma}(t) \rvert + \lvert   \varphi_{n}(t) - \varphi(t))\rvert \bigr{)},
\end{align*}
for every $t \geq t_{0}$, where $\tilde{C}>0$ is a suitable constant. Using the fact that
\begin{equation*}
\partial^{3}_{ijk} U(y) = \mathcal{O} (\lvert y \rvert^{-\alpha-3}), \quad \text{as $\lvert y \rvert \to +\infty$, $y \in \mathcal{T}(\xi^+,R/2,\eta)$}
\end{equation*}
for every $i,j,k \in \{1,\ldots,d\}$, due to the homogeneity of $U$ (compare with \eqref{cond-U0}, \eqref{cond-U1}, \eqref{cond-U2}), together with \eqref{cruciale} and \eqref{Hardy-sqrt}, we thus 
deduce the existence of $\hat{C}>0$ such that
\begin{align}\label{meanvalue}
& \lvert \nabla^{2} U(y_{\sigma_{n}}(t)+\varphi_{n}(t)) - \nabla^{2} U(y_{\sigma}(t)+\varphi(t)) \rvert \leq \\
& \leq \hat{C} t^{-\frac{2(\alpha+3)}{\alpha+2}} \biggl{(}\lvert \sigma_{n} - \sigma \rvert \max_{t \in \mathopen{[}1,2\mathclose{]}} \lvert w(t) \rvert + \Vert \varphi_{n} - \varphi \Vert t^{\frac{1}{2}} \biggr{)},
\end{align}
for every $t \geq t_{0}$. Hence, using twice \eqref{Hardy-sqrt} (for $\psi$ and $\zeta$) we obtain
\begin{align*}
& \sup_{\substack{\Vert \psi \Vert \leq 1\\  \Vert \zeta \Vert \leq 1}} \biggl{\lvert}\int_{t_{0}}^{+\infty} \bigl{\langle}  \bigl{(} \nabla^{2} U(y_{\sigma_{n}}(t)+\varphi_{n}(t)) - \nabla^{2} U(y_{\sigma}(t)+\varphi(t))\bigr{)}\psi(t),\zeta(t)\bigr{\rangle} \,\mathrm{d}t \biggr{\rvert} \leq \\
& \leq \hat{C} \int_{t_{0}}^{+\infty} t^{1-\frac{2(\alpha+3)}{\alpha+2}} \biggl{(} \lvert \sigma_{n} - \sigma \rvert \max_{t \in \mathopen{[}1,2\mathclose{]}} \lvert w(t) \rvert + \Vert \varphi_{n} - \varphi \Vert t^{\frac{1}{2}} \biggr{)} \,\mathrm{d}t \\
& = \hat{C} \biggl{(}\lvert \sigma_{n} - \sigma \rvert \max_{t \in \mathopen{[}1,2\mathclose{]}} \lvert w(t) \rvert \int_{t_{0}}^{+\infty}  t^{-\frac{\alpha+4}{\alpha+2}} \, \mathrm{d}t + \Vert \varphi_{n} - \varphi \Vert \int_{t_{0}}^{+\infty} t^{-\frac{\alpha+6}{2(\alpha+2)}} \,\mathrm{d}t \biggr{)}, 
\end{align*}
which goes to zero as $n \to \infty$ (since $\alpha<2$).

As for \eqref{secondopezzo}, we use again twice \eqref{Hardy-sqrt} (for $\psi$ and $\zeta$) to obtain
\begin{align*}
&\sup_{\substack{\Vert \psi \Vert \leq 1\\  \Vert \zeta \Vert \leq 1}} \biggl{\lvert} \int_{t_{0}}^{+\infty} \bigl{\langle}
\bigl{(} \nabla^{2} W_{\varepsilon_{n}}(t,y_{\sigma_{n}}(t)+\varphi_{n}(t)) - \nabla^{2} W_{\varepsilon}(t,y_{\sigma}(t)+\varphi(t))\bigr{)}\psi(t),\zeta(t)\bigr{\rangle} \,\mathrm{d}t \biggr{\rvert} \\
& \leq \int_{t_{0}}^{+\infty} t \lvert \nabla^{2} W_{\varepsilon_{n}}(t,y_{\sigma_{n}}(t)+\varphi_{n}(t)) - \nabla^{2} W_{\varepsilon}(t,y_{\sigma}(t)+\varphi(t)) \rvert \,\mathrm{d}t.
\end{align*}
We are going to show that the integral goes to zero, using the dominated convergence theorem. The pointwise convergence of the integrand follows from $\varphi_{n} \to \varphi$ in $\mathcal{D}^{1,2}_{0}(t_{0},+\infty)$; on the other hand, the $L^{1}$-bound follows from \eqref{We3} together with the estimate \eqref{y+phi}.

We finally claim that the linear form
$L_{\varepsilon,\sigma,\varphi}$ defined in \eqref{da2} and the bilinear form
$B_{\varepsilon,\sigma,\varphi}$ defined in \eqref{d2a2} are, respectively, the first and second Gateaux differential of $\mathcal{A}_{\varepsilon,\sigma}^{2}$ at the point $\varphi$, that is, for every $\psi \in \mathcal{D}^{1,2}_{0}(t_{0},+\infty)$,
\begin{equation}\label{gat-1}
\lim_{\vartheta \to 0} \int_{t_{0}}^{+\infty} \Biggl{(} \frac{K_{\varepsilon,\sigma}(t,\varphi(t) + \vartheta \psi(t)) - K_{\varepsilon,\sigma}(t,\varphi(t))}{\vartheta} - 
\bigl{\langle} \nabla K_{\varepsilon,\sigma}(t,\varphi(t)),\psi(t) \bigr{\rangle} \Biggr{)} \,\mathrm{d}t = 0
\end{equation}
and
\begin{align}\label{gat-2}
& \lim_{\vartheta \to 0}\sup_{\Vert \zeta \Vert \leq 1} \Biggl{\lvert}\int_{t_{0}}^{+\infty} \Biggl{(} \Biggl{\langle}\frac{\nabla K_{\varepsilon,\sigma}(t,\varphi(t) + \vartheta \psi(t)) - \nabla K_{\varepsilon,\sigma}(t,\varphi(t))}{\vartheta}, \zeta(t) \Biggr{\rangle} \\
& \hspace{165pt} - \langle \nabla^{2} K_{\varepsilon,\sigma}(t,\varphi(t))\psi(t),\zeta(t) \rangle \Biggr{)} \,\mathrm{d}t \Biggr{\rvert} = 0.
\end{align}

We begin by verifying \eqref{gat-1}. Defining, for $\vartheta \in [-1,1]$ and $t \geq t_{0}$,
\begin{equation*}
g_{\vartheta}(t) = \int_{0}^{1} \bigl{(} \nabla K_{\varepsilon,\sigma}(t,\varphi(t) + s \vartheta \psi(t)) - \nabla K_{\varepsilon,\sigma}(t,\varphi(t) ) \bigr{)}\,\mathrm{d}s,
\end{equation*}
it can be readily checked that \eqref{gat-1} equivalently reads as
\begin{equation}\label{gat-1bis}
\lim_{\vartheta \to 0} \int_{t_{0}}^{+\infty} \bigl{\langle} g_{\vartheta}(t), \psi(t) \bigr{\rangle} \,\mathrm{d}t = 0.
\end{equation}
It is easy to see that $g_{\vartheta}(t) \to 0$ as $\vartheta\to0$  for every $t \geq t_{0}$. On the other hand, using \eqref{K2} we find
\begin{equation*}
\lvert g_{\vartheta}(t) \rvert \leq \dfrac{C_{K}}{t^{2}} \int_{0}^{1} \bigl{(} \lvert \varphi(t) + s \vartheta \psi(t) \rvert + \lvert \varphi(t) \rvert \bigr{)} \,\mathrm{d}s
\leq \dfrac{C_{K}}{t^{2}} \bigl{(} 2 \lvert \varphi(t) \rvert + \lvert \psi(t) \rvert \bigr{)},
\end{equation*}
implying
\begin{equation*}
\bigl{\lvert} \bigl{\langle} g_{\vartheta}(t),\psi(t) \bigr{\rangle}\bigr{\rvert} \leq C_{K} \biggl{(} \dfrac{2\lvert \varphi(t) \rvert \lvert \psi(t) \rvert}{t^{2}} + \dfrac{\lvert\psi(t)\rvert^{2}}{t^{2}}\biggr{)},
\end{equation*}
for every $t \geq t_{0}$. By Hardy inequality \eqref{Hardy}, the right-hand side is an $L^{1}$-function; therefore the dominated convergence theorem applies yielding \eqref{gat-1bis}.

We now focus on \eqref{gat-2}. Similarly as before, we are led to verify that
\begin{equation*}
\lim_{\vartheta \to 0} \sup_{\Vert \zeta \Vert \leq 1} \int_{t_{0}}^{+\infty} \bigl{\langle} G_{\vartheta}(t) \psi(t),\zeta(t) \bigr{\rangle} \,\mathrm{d}t = 0,
\end{equation*}
where, for $\vartheta \in \mathopen{[}-1,1\mathclose{]}$ and $t \geq t_{0}$,
\begin{equation*}
G_{\vartheta}(t) = \int_{0}^{1} \Bigl{(} \nabla^{2} K_{\varepsilon,\sigma}(t,\varphi(t) + s \vartheta \psi(t)) - \nabla^{2} K_{\varepsilon,\sigma}(t,\varphi(t) ) \Bigr{)}\,\mathrm{d}s.
\end{equation*}
Again, $G_{\vartheta}(t) \to 0$ as $\vartheta\to0$ for every $t \geq t_{0}$.
Using twice \eqref{Hardy-sqrt} (for $\psi$ and $\zeta$), we find
\begin{equation*}
\sup_{\Vert \zeta \Vert \leq 1} \int_{t_{0}}^{+\infty} \bigl{\langle} G_{\vartheta}(t) \psi(t),\zeta(t) \bigr{\rangle} \,\mathrm{d}t \leq \Vert \psi \Vert \int_{t_{0}}^{+\infty} t \lvert G_{\vartheta}(t) \rvert \,\mathrm{d}t
\end{equation*}
and we can thus conclude by showing that the function $t \lvert G_{\vartheta}(t) \rvert$ is $L^{1}$-bounded. To this end, \eqref{K3} is not enough and we have to use the same strategy as for the proof of \eqref{d2a2}. Precisely, we first write
\begin{align*}
G_{\vartheta}(t) & = \int_{0}^{1} \Bigl{(} \nabla^{2} U(y_{\sigma}(t)+\varphi(t) + s \vartheta \psi(t)) - \nabla^{2} U(y_{\sigma}(t)+\varphi(t) ) \Bigr{)}\,\mathrm{d}s \\
& \quad + \int_{0}^{1} \Bigl{(} \nabla^{2} W_{\varepsilon}(t,y_{\sigma}(t)+\varphi(t) + s \vartheta \psi(t)) - \nabla^{2} W_{\varepsilon}(t,y_{\sigma}(t)+\varphi(t) ) \Bigr{)}\,\mathrm{d}s.
\end{align*}
Arguing exactly as in the proof of \eqref{meanvalue} (with $\sigma_{n} = \sigma$ and $\varphi + s\vartheta \psi$ in place of $\varphi_{n}$), 
we find on one hand that, when $\lvert \vartheta \rvert$ is small enough,
\begin{equation*}
\biggl{\lvert} \int_{0}^{1} \Bigl{(} \nabla^{2} U(y_{\sigma}(t)+\varphi(t) + s \vartheta \psi(t)) - \nabla^{2} U(y_{\sigma}(t)+\varphi(t) ) \Bigr{)}\,\mathrm{d}s \biggr{\rvert} \leq \tilde{C} t^{-\frac{2(\alpha+3)}{\alpha+2}} \Vert \psi \Vert t^{\frac{1}{2}},
\end{equation*}
for every $t \geq t_{0}$. On the other hand, using twice \eqref{y+phi} (the first time with $\varphi + s \vartheta \psi$ in place of $\varphi$) together with \eqref{We3} we find a constant $\check{C} > 0$ such that, when $\lvert \vartheta \rvert$ is small enough,
\begin{equation*}
\biggl{\lvert} \int_{0}^{1} \Bigl{(} \nabla^{2} W_{\varepsilon}(t,y_{\sigma}(t)+\varphi(t) + s \vartheta \psi(t)) - \nabla^{2} W_{\varepsilon}(t,y_{\sigma}(t)+\varphi(t) ) \Bigr{)}\,\mathrm{d}s \biggr{\rvert} \leq \check{C} t^{-\frac{2(\beta+2)}{\alpha+2}},
\end{equation*}
for every $t \geq t_{0}$. Summing up, for $\lvert \vartheta \rvert$ small enough, we find
\begin{equation*}
t \lvert G_{\vartheta}(t) \rvert \leq \tilde{C} \Vert \psi \Vert t^{-\frac{\alpha+6}{2(\alpha+2)}} + \check{C} t^{-\frac{2\beta-\alpha+2}{\alpha+2}}, \quad \text{for every $t \geq t_{0}$},
\end{equation*}
proving the desired $L^{1}$-bound (since $\alpha<2$ and $\alpha<\beta$).

We are finally ready to summarize and conclude. The existence of the limit in \eqref{gat-1} implies that the linear form
$L_{\varepsilon,\sigma,\varphi}$ defined in \eqref{da2} is the Gateaux differential of $\mathcal{A}_{\varepsilon,\sigma}^{2}$ at the point $\varphi$.
Since such a form is continuous (in $\varphi$), as proved in \eqref{L-cont}, we infer that $L_{\varepsilon,\sigma,\varphi}$ is the (Fr\'{e}chet) differential of $\mathcal{A}_{\varepsilon,\sigma}^{2}$ at the point $\varphi$ (and, moreover, $\mathcal{A}_{\varepsilon,\sigma}^{2}$ is of class $\mathcal{C}^{1}$ on the open set $\Omega_{\rho}$).

Similarly, the existence of the limit in \eqref{gat-2} implies that the bilinear form
$B_{\varepsilon,\sigma,\varphi}$ defined in \eqref{d2a2} is the Gateaux differential of $\mathrm{d} \mathcal{A}_{\varepsilon,\sigma}^{2}$ at the point $\varphi$.
Since such a form is continuous (in $\varphi$) as proved in \eqref{B-cont}, we infer that $B_{\varepsilon,\sigma,\varphi}$ is the (Fr\'{e}chet) differential of $\mathrm{d}\mathcal{A}_{\varepsilon,\sigma}^{2}$ at the point $\varphi$, that is, the second differential of $\mathcal{A}_{\varepsilon,\sigma}^{2}$ at the point $\varphi$. Hence, the functional $\mathcal{A}_{\varepsilon,\sigma}^{2}$ is of class $\mathcal{C}^{2}$ on the open set $\Omega_{\rho}$.

Recalling \eqref{split} and the discussion about $\mathcal{A}^{1}$ and $\mathcal{A}^{3}_{\varepsilon,\sigma}$ at the beginning of the proof, we 
conclude that the functional $\mathcal{A}_{\varepsilon,\sigma}$ is of class $\mathcal{C}^{2}$ on the open set $\Omega_{\rho}$. 

All this implies that $F$ is differentiable in $\varphi$, with differential 
\begin{equation*}
\mathrm{D}_{\varphi} F(\varepsilon,\sigma,\varphi) = \mathrm{d}^{2} \mathcal{A}_{\varepsilon,\sigma}(\varphi) = \mathrm{d}^{2} \mathcal{A}^{1}(\varphi) + \mathrm{d}^{2} \mathcal{A}_{\varepsilon,\sigma}^{2}(\varphi).
\end{equation*}
The continuity of $F$ and $\mathrm{D}_{\varphi} F$ with respect to the three variables $(\varepsilon,\sigma,\varphi)$ thus follows from 
\eqref{h-cont0}, \eqref{L-cont} and \eqref{B-cont}.
\end{proof}

Our goal now is to apply the implicit function theorem to the function $F$. To this end, we first observe that
\begin{equation*}
h_{0,0}(t) \equiv 0, \qquad \nabla K_{0,0}(t,0) \equiv 0,
\end{equation*}
implying $F(0,0,0) = 0$. On the other hand,
\begin{equation*}
\nabla^{2} K_{0,0}(t,0) = \nabla^{2} U(y_{0}(t)), \quad \text{for every $t \geq t_{0}$},
\end{equation*}
so that
\begin{equation}\label{diff-for}
\mathrm{D}_{\varphi} F(0,0,0)[\psi,\zeta] = \int_{t_{0}}^{+\infty} \Bigl{(} \bigl{\langle}  \dot{\psi}(t),\dot{\zeta}(t) \bigr{\rangle} +  \bigl{\langle} \nabla^{2} U(y_{0}(t))\psi(t),\zeta(t) \bigr{\rangle} \Bigr{)}\,\mathrm{d}t,
\end{equation}
for every $\psi,\zeta \in \mathcal{D}^{1,2}_{0}(t_{0},+\infty)$. 
In the above formula, we have meant the differential $\mathrm{D}_{\varphi} F(0,0,0)$ as a continuous bilinear form on $\mathcal{D}^{1,2}_{0}(t_{0},+\infty)$, using the canonical isomorphism between bilinear forms on a Banach space $X$ and linear operators from $X$ to $X^{*}$.
Notice that the invertibility of $\mathrm{D}_{\varphi} F(0,0,0)$ (as a linear operator from $\mathcal{D}^{1,2}_{0}(t_{0},+\infty)$ to its dual) is equivalent 
to the fact that for every $T \in \bigl{(}\mathcal{D}^{1,2}_{0}(t_{0},+\infty)\bigr{)}^{\!*}$ there exists $\psi_T \in \mathcal{D}^{1,2}_{0}(t_{0},+\infty)$ such that
\begin{equation*}
\mathrm{D}_{\varphi} F(0,0,0)[\psi_T,\zeta] = T[\zeta], \quad \text{for every $\zeta \in \mathcal{D}^{1,2}_{0}(t_{0},+\infty)$}.
\end{equation*}
To show that this is true, we are going to use the Lax--Milgram theorem, by proving that the quadratic form associated to $\mathrm{D}_{\varphi} F(0,0,0)$ is coercive. This is the content of the next proposition.

\begin{proposition}\label{prop-3.2}
There exists $\kappa>0$ such that for every $\psi \in \mathcal{D}^{1,2}_{0}(t_{0},+\infty)$ it holds that
\begin{equation*}
\int_{t_{0}}^{+\infty} \Bigl{(} \lvert \dot{\psi}(t) \rvert^{2} +  \bigl{\langle} \nabla^{2} U(y_{0}(t))\psi(t),\psi(t) \bigr{\rangle} \Bigr{)}\,\mathrm{d}t \geq \kappa \Vert \psi \Vert^{2}.
\end{equation*}
\end{proposition}

\begin{proof}
Let $\psi \in \mathcal{D}^{1,2}_{0}(t_{0},+\infty)$. From \eqref{def-y0}, by using the $(-\alpha-2)$-homogeneity of $\nabla^{2} U$, we deduce that
\begin{equation*}
\bigl{\langle} \nabla^{2} U(y_{0}(t))\psi(t),\psi(t) \bigr{\rangle} 
= \omega^{-\alpha-2} t^{-2} \bigl{\langle} \nabla^{2} U(\xi^{+})\psi(t),\psi(t) \bigr{\rangle},
\quad \text{for every $t\geq t_{0}$.}
\end{equation*}
By exploiting the decomposition
\begin{equation*}
\psi(t)= v(t) + \lambda(t) \xi^{+}, \quad \text{with $v(t)\in\mathcal{T}_{\xi^{+}}\mathcal{E}$ and $\lambda(t)\in\mathbb{R}$,}
\end{equation*}
(here, $\mathcal{T}_{\xi^{+}}\mathcal{E}$ denotes the tangent space of $\mathcal{E}$ at $\xi^{+}$)
and formula \eqref{eq-diffseccc}, we have
\begin{align*}
&\bigl{\langle} \nabla^{2} U(\xi^{+})\psi(t),\psi(t) \bigr{\rangle} = 
\\
&= \bigl{\langle} \nabla^{2} U(\xi^{+})v(t),v(t) \bigr{\rangle}
+ 2 \lambda(t) \bigl{\langle} \nabla^{2} U(\xi^{+})\xi^{+},v(t) \bigr{\rangle}
+ (\lambda(t))^{2} \bigl{\langle} \nabla^{2} U(\xi^{+})\xi^{+},\xi^{+} \bigr{\rangle}
\\
&= \bigl{\langle} \nabla^{2} U(\xi^{+})v(t),v(t) \bigr{\rangle}
+ 2 \alpha (\alpha+1) \lambda(t) \, U(\xi^{+}) \langle \xi^{+},v(t) \rangle
\\
& \quad 
+ \alpha (\alpha+1) (\lambda(t))^{2} U(\xi^{+}) \bigl{\langle} \xi^{+},\xi^{+} \bigr{\rangle}
\\
&= \bigl{\langle} \nabla^{2} U(\xi^{+})v(t),v(t) \bigr{\rangle}
+ \alpha (\alpha+1) (\lambda(t))^{2} U(\xi^{+}),
\quad \text{for every $t\geq t_{0}$,}
\end{align*}
where the last equality follows from the facts that $\langle \xi^{+},v(t) \rangle=0$ and $\lvert \xi^{+} \rvert =1$.
By arguing as in the proof of \cite[Proposition~16]{Mo-notes} (dealing with the Newtonian $N$-body potential), recalling that $\mathcal{U}=U|_{\mathcal{E}}$, we infer that
\begin{equation}\label{Moeckel}
\bigl{\langle} \nabla^{2} U(\xi^{+})v(t),v(t) \bigr{\rangle} 
= \bigl{\langle} \nabla^{2} \mathcal{U}(\xi^{+})v(t),v(t) \bigr{\rangle} - \alpha \, \mathcal{U}(\xi^{+}) \lvert v(t) \rvert^{2},
\end{equation}
for every $t\geq t_{0}$.

From the above discussion, we obtain that
\begin{align*}
&\bigl{\langle} \nabla^{2} U(y_{0}(t))\psi(t),\psi(t) \bigr{\rangle} =
\\
&= \dfrac{\omega^{-\alpha-2}}{t^{2}} \Bigl{(} \bigl{\langle} \nabla^{2} U(\xi^{+})v(t),v(t) \bigr{\rangle}
+ \alpha (\alpha+1) (\lambda(t))^{2} U(\xi^{+})\Bigr{)}
\\
&= \dfrac{\omega^{-\alpha-2}}{t^{2}} \Bigl{(} \bigl{\langle} \nabla^{2} \mathcal{U}(\xi^{+})v(t),v(t) \bigr{\rangle} 
- \alpha \, \mathcal{U}(\xi^{+}) \lvert v(t) \rvert^{2} + \alpha (\alpha+1) (\lambda(t))^{2} \, \mathcal{U}(\xi^{+}) \Bigr{)},
\end{align*}
for every $t\geq t_{0}$. 
Condition $(\textsc{BS})$ implies that there exists $\zeta\in\mathopen{]}0,(2-\alpha)^{2}/8\mathclose{[}$ such that
\begin{equation*}
\bigl{\langle} \nabla^{2} \mathcal{U}(\xi^{+})v(t),v(t) \bigr{\rangle} \geq \biggl{(} - \dfrac{(2-\alpha)^{2}}{8}+\zeta \biggr{)} \mathcal{U}(\xi^{+}) \lvert v(t) \rvert^{2},
\end{equation*}
for every $t\geq t_{0}$. Therefore, we have
\begin{align*}
&\bigl{\langle} \nabla^{2} U(y_{0}(t))\psi(t),\psi(t) \bigr{\rangle} \geq
\\
&\geq \dfrac{\omega^{-\alpha-2}}{t^{2}} \biggl{[} \biggl{(} - \dfrac{(2-\alpha)^{2}}{8}+\zeta \biggr{)} \mathcal{U}(\xi^{+}) \lvert v(t) \rvert^{2}
- \alpha \, \mathcal{U}(\xi^{+}) \lvert v(t) \rvert^{2}
\\
& \hspace{181pt}  + \alpha (\alpha+1) (\lambda(t))^{2} \, \mathcal{U}(\xi^{+}) \biggr{]}
\\
& = \dfrac{\omega^{-\alpha-2}}{t^{2}} \, \mathcal{U}(\xi^{+}) \biggl{[} \biggl{(}-\dfrac{(2-\alpha)^{2}}{8}+\zeta-\alpha\biggr{)} \lvert v(t) \rvert^{2} + \alpha (\alpha+1) (\lambda(t))^{2}  \biggr{]}
\\
&=  \dfrac{2}{(\alpha+2)^{2} \, \mathcal{U}(\xi^{+}) t^{2}} \, \mathcal{U}(\xi^{+}) \biggl{[} \biggl{(}- \dfrac{(\alpha+2)^{2}}{8} +\zeta \biggr{)} \lvert v(t) \rvert^{2}+  \alpha (\alpha+1) (\lambda(t))^{2}  \biggr{]}
\\
&\geq  \dfrac{\lvert v(t) \rvert^{2}}{t^{2}}  \biggl{(} -\dfrac{1}{4} + \dfrac{2\zeta }{(\alpha+2)^{2}} \biggr{)}
\\
&\geq  \dfrac{\lvert \psi(t) \rvert^{2}}{t^{2}}  \biggl{(} -\dfrac{1}{4} + \dfrac{2\zeta }{(\alpha+2)^{2}} \biggr{)},
\end{align*}
for every $t\geq t_{0}$. 

By Hardy inequality \eqref{Hardy}, we finally obtain
\begin{equation*}
\int_{t_{0}}^{+\infty} \Bigl{(} \lvert \dot{\psi}(t) \rvert^{2} +  \bigl{\langle} \nabla^{2} U(y_{0}(t))\psi(t),\psi(t) \bigr{\rangle} \Bigr{)}\,\mathrm{d}t \geq \dfrac{8\zeta }{(\alpha+2)^{2}} \int_{t_{0}}^{+\infty} \lvert \dot{\psi}(t) \rvert^{2} \,\mathrm{d}t,
\end{equation*}
as desired.
\end{proof}

Summing up, by the implicit function theorem, there exist $\varepsilon^{*} \in \mathopen{]}0,1\mathclose{[}$ and $r^{*} \in \mathopen{]}0,r\mathclose{[}$ such that, 
for every $\varepsilon \in \mathopen{[}0,\varepsilon^{*}\mathclose{[}$ and for every $\sigma \in B_{r^{*}}$, there exists a solution $\varphi \in \Omega_{\rho}$
of the equation
\begin{equation*}
F(\varepsilon,\sigma,\varphi) = 0.
\end{equation*}
This means that $\varphi$ is a critical point of the action functional $\mathcal{A}_{\varepsilon,\sigma}$.
Since the space $\mathcal{C}^{\infty}_{\mathrm{c}}(\mathopen{]}t_{0},+\infty\mathclose{[})$ is contained in $\mathcal{D}^{1,2}_{0}(t_{0},+\infty)$, we deduce that $\varphi$ is a solution of equation \eqref{eq-phi} in the sense of distributions. By a standard regularity argument, $\varphi \in \mathcal{C}^{2}(\mathopen{[}t_{0},+\infty\mathclose{[})$ and solves the equation in the classical sense.

We summarize the above discussion in the final proposition of this section.

\begin{theorem}\label{th-3.1}
There exist $\varepsilon^{*} \in \mathopen{]}0,1\mathclose{[}$ and $r^{*}\in \mathopen{]}0,r\mathclose{[}$ such that, 
for every $\varepsilon \in \mathopen{[}0,\varepsilon^{*}\mathclose{[}$ and $\sigma \in B_{r^{*}}$,
there exists $\varphi \in \mathcal{C}^{2}(\mathopen{[}t_{0},+\infty\mathclose{[}) \cap \mathcal{D}^{1,2}_{0}(t_{0},+\infty)$ solution of \eqref{eq-phi}.
\end{theorem}

\begin{remark}\label{rem-3.1}
It is worth noticing that, due to Proposition~\ref{prop-3.2}, it easily follows that $\varphi$ is a non-degenerate local minimum for the corresponding action functional $\mathcal{A}_{\varepsilon,\sigma}$.
\end{remark}

\subsection{Conclusion of the proof}\label{section-5.3}

We are now in a position to conclude the proof of Theorem~\ref{th-main}. To this end, let us consider the numbers $\varepsilon^{*}$ and $r^{*}$ given in Theorem~\ref{th-3.1} and set
\begin{equation*}
R' =  \max\Bigl{\{} (\varepsilon^{*})^{-\frac{3}{2+\alpha}} \sqrt{\omega^{2}+(r^{*})^{2}} , R\Bigr{\}}, \qquad \eta' = \max\Biggl{\{} \eta, \dfrac{\omega}{\sqrt{\omega^{2}+(r^{*})^{2}}} \Biggr{\}};
\end{equation*}
in such a way $\eta' \in \mathopen{[}\eta,1\mathclose{[}$ and $\mathcal{T}(\xi^+,R',\eta') \subseteq \mathcal{T}(\xi^+,R,\eta)\subseteq\Sigma$.

Let us fix $x_{0}\in \mathcal{T}(\xi^+,R',\eta')$. We claim that there exist $\varepsilon \in \mathopen{]}0,\varepsilon^{*}\mathclose{[}$ and $\sigma \in B_{r^{*}}$ such that
\begin{equation} \label{eq-x0fin}
x_{0}= \varepsilon^{-\frac{3}{2+\alpha}} \, \bigl{(} \omega \xi^{+} +\sigma \bigr{)}.
\end{equation}
Indeed, we set
\begin{equation*}
\varepsilon = \biggl{(} \dfrac{\omega}{\langle x_{0}, \xi^{+} \rangle} \biggr{)}^{\!\frac{2+\alpha}{3}}, \qquad \sigma = \omega \biggl{(} \dfrac{x_{0}}{\langle x_{0},\xi^{+} \rangle} - \xi^{+} \biggr{)}.
\end{equation*}
Then, $\langle \sigma, \xi^{+} \rangle = 0$ so that, using the facts that $\langle x_{0}, \xi^{+} \rangle > \eta' \lvert x_{0} \rvert$ and $\lvert x_{0} \rvert>R'$, we obtain
\begin{align*}
&\lvert \sigma \rvert^{2} 
= \omega^{2} \biggl{(} \dfrac{\lvert x_{0} \rvert^{2}}{\langle x_{0}, \xi^{+} \rangle^{2}} + \lvert \xi^{+} \rvert^{2} - 2 \biggr{)} 
<  \omega^{2} \biggl{(} \dfrac{1}{(\eta')^{2}} -1 \biggr{)} \leq (r^{*})^{2},
\\
&\varepsilon <  \biggl{(} \dfrac{\omega}{\eta' \lvert x_{0} \rvert} \biggr{)}^{\!\frac{2+\alpha}{3}} 
<  \biggl{(} \dfrac{\omega}{\eta' R'} \biggr{)}^{\!\frac{2+\alpha}{3}}
\leq \varepsilon^{*},
\end{align*}
as desired.

Now, given $\varepsilon$ and $\sigma$ as in \eqref{eq-x0fin}, let us consider the solution $\varphi$ of \eqref{eq-phi} as in Theorem~\ref{th-3.1} and define
\begin{equation*}
y(t)=y_{\sigma}(t) + \varphi(t),\quad \text{for every $t \geq t_{0}$,}
\end{equation*}
as in \eqref{def-y}. As already observed in Section~\ref{section-4.1}, since $\varphi$ is a solution of \eqref{eq-phi} on $\mathopen{[}t_{0},+\infty\mathclose{[}$ the function $y$ is a solution of \eqref{eq-y0} on $\mathopen{[}t_{0},+\infty\mathclose{[}$ and the function $x$ defined by
\begin{equation} \label{eq-cambioxy}
x(t)= \varepsilon^{-\frac{3}{2+\alpha}} \, y\bigl{(}\varepsilon^{\frac{3}{2}}t+t_{0}\bigr{)}, \quad \text{for every $t \geq 0$,}
\end{equation}
as in \eqref{def-xy}, is a solution of \eqref{eq-main} on $\mathopen{[}0,+\infty\mathclose{[}$. We also notice that
\begin{equation*}
x(0)= \varepsilon^{-\frac{3}{2+\alpha}} \, y(t_{0}) = \varepsilon^{-\frac{3}{2+\alpha}} \, \bigl{(} \omega \xi^{+} +\sigma \bigr{)}=x_{0},
\end{equation*}
by \eqref{eq-x0fin}. Moreover, recalling that $y_{\sigma}(t) = y_{0}(t)$ for $t \geq t_{0}+1$ and using inequality \eqref{Hardy-sqrt}, it holds that
$\lvert y(t) \rvert \sim \omega t^{\frac{2}{\alpha+2}}$ for $t \to +\infty$, implying
$\lvert x(t) \rvert \sim \omega t^{\frac{2}{\alpha+2}}$ for $t \to +\infty$, as well.

In order to conclude the proof we thus need to show 
that $x$ is asymptotic to the central configuration $\xi^{+}$ and that $\dot x(t) \to 0$ for $t \to +\infty$. Recalling the change of variables \eqref{eq-cambioxy}, it is immediate to see that this is the case if and only if the corresponding properties are verified by $y$, that is,
\begin{equation} \label{eq-limitey1}
\lim_{t\to+\infty} \dfrac{y(t)}{|y(t)|} = \xi^{+}
\end{equation}
and
\begin{equation} \label{eq-limitey2}
\displaystyle \lim_{t\to+\infty} \dot{y}(t) = 0.
\end{equation}

As far as \eqref{eq-limitey1} is concerned, we recall again that $y_{\sigma}(t) = y_{0}(t)$ for $t \geq t_{0}+1$, implying
\begin{align*}
\dfrac{y(t)}{|y(t)|} & = \dfrac{y_{0}(t)+\varphi(t)}{|y_{0}(t)+\varphi(t)|}
\displaystyle = \dfrac{y_{0}(t)+\varphi(t)}{ \bigl{(}|y_{0}(t)|^{2} + |\varphi(t)|^{2} + 2 \langle y_{0}(t),\varphi(t) \rangle \bigr{)}^{\!\frac{1}{2}} } \\ 
& = \dfrac{y_{0}(t)+\varphi(t)}{|y_{0}(t)| \biggl{(} 1+ \dfrac{|\varphi(t)|^{2}}{|y_{0}(t)|^{2}} + \dfrac{2\langle y_{0}(t),\varphi(t) \rangle}{|y_{0}(t)|^{2}}\biggr{)}^{\!\frac{1}{2}}} \\
& = \biggl{(}\dfrac{y_{0}(t)}{|y_{0}(t)|}+ \dfrac{\varphi(t)}{|y_{0}(t)|}\biggr{)}
\biggl{(} 1+ \dfrac{|\varphi(t)|^{2}}{|y_{0}(t)|^{2}} + \dfrac{2 \langle y_{0}(t),\varphi(t) \rangle}{|y_{0}(t)|^{2}}\biggr{)}^{\!-\frac{1}{2}},
\end{align*}
for every $t \geq t_{0}+1$. Recalling that $y_{0}(t)=\omega t^{\frac{2}{\alpha+2}} \xi^{+}$ and using \eqref{Hardy-sqrt}, we infer that
\begin{equation*}
\frac{y_{0}(t)}{\lvert y_{0}(t) \rvert} = \xi^{+} \quad \text{ and } \quad \displaystyle \lim_{t\to+\infty} \dfrac{\varphi(t)}{|y_{0}(t)|} = 0,
\end{equation*}
thus concluding that
\begin{equation*}
\lim_{t\to+\infty} \dfrac{y(t)}{|y(t)|} = \xi^{+}.
\end{equation*}

Finally, we prove the validity of \eqref{eq-limitey2}.
Using \eqref{Ue2} and \eqref{y+phi}, we find
\begin{equation*}
\vert \ddot y(t) \vert \leq \frac{2^{\alpha+1}C''}{\omega^{\alpha+1}} \, t^{-\frac{2(\alpha+1)}{\alpha+2}}, \quad \text{for every $t \geq t_{0}$.}
\end{equation*}
Hence, $\ddot y$ is integrable on $\mathopen{[}t_{0},+\infty\mathclose{[}$ and thus that there exists $\ell\in \mathbb{R}^d$ such that
\begin{equation*}
\lim_{t\to+\infty} \dot{y}(t)=\ell.
\end{equation*}
Since $y(t) = y_0(t) + \varphi(t)$ for $t \geq t_{0}+1$, we infer that
$\lim_{t \to +\infty}\dot\varphi(t) = \ell$, as well. 
On the other hand, since $\varphi\in\mathcal{D}^{1,2}_{0}(t_{0},+\infty)$ we have $\int_{t_{0}}^{+\infty}\lvert\dot{\varphi}(t)\rvert^{2} \,\mathrm{d}t<+\infty$,
implying
\begin{equation*}
\liminf_{t\to+\infty} \, \lvert \dot{\varphi}(t)\rvert = 0.
\end{equation*}
Then, we deduce that $\ell = 0$, thus proving \eqref{eq-limitey2}.

\section{Some applications}\label{section-6}

In this section we present some applications of Theorem~\ref{th-main} to various classical problems of Celestial Mechanics.
From now on, we will use the symbol $\lVert \cdot \rVert$ to denote the Euclidean norm of a vector in $\mathbb{R}^k$ (the specific value of $k$ being clear from the context).

\subsection{The $N$-body problem}\label{section-6.1}

As a first application, we deal with the classical $N$-body problem
\begin{equation*}
m_{i} \ddot q_{i} = -\sum_{j \neq i} \frac{m_{i} m_{j} (q_{i} - q_{j})}{\|q_{i} - q_{j} \|^3}, \qquad i=1,\ldots,N,
\end{equation*}
where $m_i > 0$ and $q_i\in \mathbb{R}^k$, for $i=1,\ldots,N$ and $k \geq 2$. As already observed in the introduction, this problem fits the general framework of equation \eqref{eq-main} with $d = kN$, 
\begin{equation}\label{eq-oggi10}
(\mu_1,\ldots,\mu_d) = (\overbrace{m_1,\ldots,m_1}^{\text{$k$ times}},m_2,\ldots,m_2,\ldots,m_N,\ldots,m_N),
\end{equation}
$W\equiv 0$ and
\begin{equation*}
U(x) = \sum_{i<j}\frac{m_i m_{j}}{\|q_i - q_{j}\|},
\qquad
\text{$x = (q_1,\ldots,q_N)$.}
\end{equation*}
Notice that in this situation the potential $U$ is defined on the set 
\begin{equation*}
\Sigma = \bigl{\{} (q_1,\ldots,q_N)\in \mathbb{R}^{kN} \colon \text{$q_i\neq q_j$ for $i\neq j$}\bigr{\}}.
\end{equation*}

Theorem~\ref{th-main} applies to any minimal central configuration (as the main result in \cite{MaVe-09}) and to any central configuration satisfying the $(\textsc{BS})$-condition; for instance, following \cite[Example~3.12]{BaHuPoTe-PP}, it applies to the collinear configuration for three bodies, two of mass $1$ and one of mass $m\in (0,27/4)$.

\subsection{The $N$-centre problem and the restricted $(N+1)$-body problem}\label{section-6.2}

As a next application, we consider the equation
\begin{equation}\label{eq-centri}
\ddot x = -\sum_{i=1}^{N} \frac{m_{i} (x-c_{i}(t))}{\|x - c_{i}(t) \|^{3}}, \qquad x \in \mathbb{R}^d,
\end{equation}
where $m_{i} > 0$ and $c_{i} \colon \mathbb{R} \to \mathbb{R}^d$ are continuous functions for $i=1,\ldots,N$.
Equation~\eqref{eq-centri} models the motion of a zero-mass particle $x$ under the Newtonian attraction of $N$ moving bodies $c_{i}$ of mass $m_{i}$. The case when all the bodies $c_{i}$ are fixed, namely $c_{i}(t) \equiv c_{i}$ for every $i =1,\ldots,N$, is usually referred to as $N$-centre problem (see \cite{BoDaTe-17} for some bibliography on the problem). For general moving bodies $c_{i}$, we have the following result (in the statement, $\mathbb{S}^{d-1} = \{x \in \mathbb{R}^d \colon \lVert x \rVert = 1 \} = \mathcal{E}$).

\begin{corollary}\label{cor-main}
Let us suppose that
\begin{equation}\label{hp-cor}
\sup_{t \in \mathbb{R}} \, \lVert c_{i}(t)  \rVert < +\infty, \quad \text{for every $i=1,\ldots,N$.}
\end{equation}
Then, for every $\xi^{+} \in \mathbb{S}^{d-1}$, there exist $R' > \sup_{t \in \mathbb{R}} \lVert c_{i}(t)  \rVert$ and $\eta' \in \mathopen{]}0,1\mathclose{[}$ such that, for every $x_{0} \in \mathcal{T}(\xi^{+},R',\eta')$, there exists a parabolic solution $x \colon \mathopen{[}0,+\infty\mathclose{[} \to \mathbb{R}^{d}$ of equation \eqref{eq-centri}, satisfying $x(0) = x_{0}$ and 
\begin{equation*}\lim_{t\to+\infty} \dfrac{x(t)}{\|x(t)\|} = \xi^{+}.
\end{equation*}
Moreover, $\lVert x(t) \rVert \sim \omega t^{\frac{2}{3}}$ for $t \to +\infty$, where $\omega = \sqrt[3]{\frac{9}{2} m}$
and $m = \sum_{i=1}^{N} m_{i}$. 
\end{corollary}

\begin{proof}
We are going to check that the assumptions of Theorem~\ref{th-main} are satisfied (in view of Remark \ref{rem-moregeneral}, we can assume that the functions $c_i$ are just continuous).
	
To this end, we first observe that equation \eqref{eq-centri} can be written in the form \eqref{eq-main} with $\mu_j=1$ for $j=1,\ldots, d$,
$U(x)=m/\|x\|$ for $x \in \Sigma = \mathbb{R}^d \setminus \{0\}$ and
\begin{equation*}
W(t,x) = \sum_{i=1}^{N} \frac{m_{i}}{\|x - c_{i}(t)\|}- \frac{m}{\|x\|}, \qquad \text{$x\in\mathbb{R}^{d}$ with $\lVert x \rVert > R = \sup_{t \in \mathbb{R}} \|c_i(t)\| + 1$.} 
\end{equation*}
Let us observe that in this situation we have $\lvert \cdot \rvert = \lVert \cdot \rVert$. Moreover, for the potential $U$ any configuration $\xi^{+}\in \mathbb{R}^d$ with $\lVert \xi^+\rVert=1$ is a minimizing central configuration (indeed, $U$ is constant on $\mathbb{S}^{d-1}$). As far as assumption \eqref{hp-main} is concerned, let us first define $g(t) = \sum_{i=1}^{N} m_{i} c_{i}(t)$, for every $t\in \mathbb{R}$ and let us denote by $x \otimes y$ the square matrix of components $(x \otimes y)_{ij} = x_{i} y_{j}$, for $i,j=1,\ldots, d$; then from \eqref{hp-cor} we deduce that
\begin{align*}
& W(t,x) = \dfrac{\langle g(t),x\rangle}{\|x\|^{3}} + \mathcal{O}\biggl{(} \frac{1}{\|x\|^{3}}\biggr{)}, \\
& \nabla W(t,x) = \frac{g(t)}{\|x\|^{3}} - 3 \frac{\langle g(t),x \rangle x}{\|x\|^{5}} + \mathcal{O}\biggl{(}\frac{1}{\|x\|^{4}}\biggr{)}, \\
& \nabla^{2} W(t,x) = -6 \, \frac{g(t) \otimes x}{\|x\|^{5}} - 3 \, \frac{\langle g(t),x \rangle}{\|x\|^{5}} \textrm{Id}_{\mathbb{R}^d} + 15 \, \frac{\langle g(t),x \rangle}{\|x\|^7} x \otimes x +\mathcal{O}\biggl{(}\frac{1}{\|x\|^{5}}\biggr{)},
\end{align*}
for $\|x\| \to +\infty$, uniformly in $t \in \mathbb{R}$. 
Using once more assumption \eqref{hp-cor} together with elementary linear algebra inequalities, we see that condition \eqref{hp-main} is satisfied with $\beta = 2$. Therefore, Theorem~\ref{th-main} can be applied, yielding the conclusion.
\end{proof}

Let us notice that the elliptic restricted (planar) three-body problem is a particular case of equation \eqref{eq-centri},
where $d =2$, $N=2$ and 
\begin{equation*}
m_1 = \mu, \quad m_2 = 1-\mu, \quad c_1(t) = -\mu q_{0}(t), \quad c_2(t) = (1-\mu) q_{0}(t),
\end{equation*}
with $\mu \in \mathopen{]}0,1\mathclose{[}$ and $q_{0}$ a $2\pi$-periodic function (see \cite{GuSeMaSa-17} for more details). Of course, condition \eqref{hp-cor} is satisfied due to the periodicity, so that Corollary~\ref{cor-main} straightly applies. The elliptic restricted spatial 
three-body problem could be treated in the same manner (simply, $d=3$). More in general, for any bounded motion of $N$ moving bodies $c_i$, a restricted $(N+1)$-body problem can be described by equation \eqref{eq-centri} and Corollary \ref{cor-main} straightly applies.

\subsection{The restricted $(N+H)$-body problem}\label{section-6.3}

As a last application, we deal with the equation
\begin{equation}\label{eq-centri2}
m_j \ddot q_j = -\sum_{i=1}^{N} \frac{\widetilde{m}_{i} m_j (q_j-c_{i}(t))}{\|q_j - c_{i}(t) \|^{3}}-\sum_{l \neq j} \frac{m_{l} m_{j} (q_{l} - q_{j})}{\|q_{l} - q_{j} \|^3}, \qquad j=1,\ldots ,H,
\end{equation}
where $m_j>0$ and $q_j \in \mathbb{R}^{k}$, for $j=1,\ldots, H$, $\widetilde{m}_{i} > 0$ and $c_{i} \colon \mathbb{R} \to \mathbb{R}^k$ are continuous functions, for $i=1,\ldots,N$. We assume $H \geq 2$, $N \geq 2$ and $k \geq 2$.

Equation \eqref{eq-centri2} models the motion of $H$ bodies $q_j$ with mass $m_j$, subjected to the reciprocal Newtonian attraction as well as the attraction of $N$ moving bodies $c_i$ with mass $\widetilde{m}_{i}$, whose motion however is not affected by the bodies $q_j$.
We can interpret this equation as a limiting model for the motion of $N+H$ bodies, when $N$ of them have big mass (the primaries) while the masses of the other $H$ bodies go to zero. For this reason, we refer to equation \eqref{eq-centri2} as a restricted $(N+H)$-body problem; notice, indeed, that in the case $H=1$ the second term in the right-hand side of equation \eqref{eq-centri2} disappears, so that we end up with equation 
\eqref{eq-centri}.

Let us notice that \eqref{eq-centri2} can be written in the form \eqref{eq-main}, where $d=kH$,
\begin{equation*}
(\mu_1,\ldots,\mu_d) = (\overbrace{m_1,\ldots,m_1}^{\text{$k$ times}},m_2,\ldots,m_2,\ldots,m_H,\ldots,m_H),
\end{equation*}
and
the functions $U$ and $W$ are given by
\begin{equation*}
U(x) = \widetilde{m} \sum_{j=1}^H \frac{m_j}{\|q_j\|} + \sum_{l<j}\frac{m_l m_{j}}{\|q_l - q_{j}\|},
\end{equation*}
and 
\begin{equation*}
W(t,x) = \sum_{j=1}^{H}\sum_{i=1}^{N} \frac{\widetilde{m}_i m_{j}}{\| q_j - c_{i}(t) \|}- \widetilde{m} \sum_{j=1}^H \frac{m_j}{\|q_j\|},
\end{equation*}
for $x = (q_1,\ldots,q_H)$ and $\widetilde{m}=\sum_{i=1}^N \widetilde{m}_i$. Notice that the natural domain of $U$ is here given by
\begin{equation*}
\Sigma = \bigl{\{} (q_1,\ldots,q_H)\in \mathbb{R}^{kH} \colon \text{$q_j \neq 0$, $q_l\neq q_j$ for $l\neq j$} \bigr{\}}.
\end{equation*}

Given $\xi^+\in \mathcal{E}$ a central configuration for $U$, and assuming that there exists $\Xi>0$ such that
\begin{equation}\label{hp-cor2}
\sup_{t \in \mathbb{R}} \|c_{i}(t)\| < \Xi, \quad \text{for every $i=1,\ldots,N$,}
\end{equation}
we claim that there exist $R > 0$ and $\eta \in \mathopen{]}0,1\mathclose{[}$ such that $W$ is defined in $\mathrm{cl}(\mathcal{T}(\xi^+,R, \eta)) \subseteq \Sigma$. To prove this, we first introduce the cone 
\begin{equation*}
\mathcal{C}(\xi^+,\eta) = \bigl{\{} x \in \mathbb{R}^{kH} \setminus \{0\} \colon \langle x, \xi^+ \rangle \geq \eta | x | \bigr{\}}.
\end{equation*}
By continuity, we can fix $\eta \in \mathopen{]}0,1\mathclose{[}$ such that $\mathcal{C}(\xi^+,\eta) \subseteq \Sigma$.
As a next step, we define the $0$-homogeneous functions $r_j(x) = \lVert q_j \rVert / \lvert x \rvert$ for $x \in \mathcal{C}(\xi^+,\eta)$.
Then
\begin{equation}\label{stima0omogeneo}
0 < \min_{\substack{x \in \mathcal{C}(\xi^+,\eta) \\ \lvert x \rvert = 1}} r_j(x)\leq r_j(x) \leq \max_{\substack{x \in \mathcal{C}(\xi^+,\eta) \\ \lvert x \rvert = 1}} r_j(x), \quad \text{for every $x \in \mathcal{C}(\xi^+,\eta)$,}
\end{equation}
so that we can fix $R > 0$ so large that, for $x \in \mathcal{C}(\xi^+,\eta)$, $\lvert x \rvert > R$ implies $\lVert q_j \rVert > \Xi + 1$ for every $j=1,\ldots,H$. The claim is thus proved.

We now check that assumption \eqref{hp-main} is satisfied. To this end, let us write
\begin{equation*}
W(t,x)=\sum_{j=1}^H W_j(t,q_j),
\end{equation*}
where 
\begin{equation*}
W_j(t,q_j)=m_j \left(\sum_{i=1}^N\frac{\widetilde{m}_{i}}{\| q_j - c_{i}(t) \|}-  \frac{\widetilde{m}}{\|q_j\|}\right), \qquad j=1, \ldots, H.
\end{equation*}
By arguing as in the proof of Corollary~\ref{cor-main}, it is possible to show that 
\begin{equation} \label{eq-oggi11}
\lvert W_j(t,q_j) \rvert + \|q_j \|\, \|D W_j(t,q_j)\| + \|q_j\|^{2} \|D^{2} W_j(t,q_j)\| = \mathcal{O}\bigl{(}\|q_j\|^{-2} \bigr{)},
\end{equation}
for $\|q_j\|\to +\infty$, $j=1,\ldots, H$, uniformly in $t\in \mathbb{R}$, where $D$ and $D^2$ stand for the Euclidean gradient and for the Hessian matrix, respectively. On the other hand, by \eqref{stima0omogeneo} and the fact that $\lVert\cdot\rVert$ and $\lvert\cdot\rvert$ are equivalent norms in $\mathbb{R}^d$ it follows that $\|x\|\to +\infty$ with 
 $x \in \mathcal{T}(\xi^+, R, \eta)$ implies that $\|q_j\|\to +\infty$ for every $j$. Hence
condition \eqref{eq-oggi11} yields
\begin{equation*}
\lvert W(t,x) \rvert + \|x\|\, \|D W(t,x)\| + \|x\|^{2} \|D^{2} W(t,x)\| = \mathcal{O}\bigl{(}\|x\|^{-2} \bigr{)},
\end{equation*}
for $\|x\|\to +\infty$ with $x \in \mathcal{T}(\xi^+, \Xi+1, \eta)$, uniformly in $t\in \mathbb{R}$. Using again the equivalence of $\lVert\cdot\rVert$ and $\lvert\cdot\rvert$ together with the facts that $\nabla = M^{-1} D$, $\nabla^2 = M^{-1}D^2$ (with $M$ the diagonal matrix with entries $\mu_1,\ldots,\mu_d$), we conclude that assumption \eqref{hp-main} is satisfied.

Hence, Theorem~\ref{th-main} applies and the following result holds true.

\begin{corollary}\label{cor-main2}
Assume condition \eqref{hp-cor2} and let $\xi^+\in \mathcal{E}$ be a central configuration for $U$ satisfying the $(\textsc{BS})$-condition. 
Then, there exist $R'>\Xi+1$ and $\eta' \in \mathopen{]}0,\eta\mathclose{[}$ such that, for every $x_{0} \in \mathcal{T}(\xi^{+},R',\eta')$, there exists a parabolic solution $x \colon \mathopen{[}0,+\infty\mathclose{[} \to \mathbb{R}^{d}$ of equation \eqref{eq-centri}, satisfying $x(0) = x_{0}$ and 
\begin{equation*}\lim_{t\to+\infty} \dfrac{x(t)}{\lvert x(t) \rvert} = \xi^{+}.
\end{equation*}
Moreover, $\lvert x(t) \rvert \sim \omega t^{\frac{2}{3}}$ for $t \to +\infty$, where $\omega=\sqrt[3]{9U(\xi^{+}) /2}$. 
\end{corollary}

We end this section by presenting an interesting example of a central configuration for $U$ satisfying the $(\textsc{BS})$-condition.

Let us assume $m_j = m$ for every $j=1,\ldots,H$ and, for simplicity, $\widetilde{m}=1$; notice that in such a situation we have
\begin{equation*}
\langle x,y \rangle =
m \, x\cdot y,\quad |x|^2=m \|x\|^2,\quad \text{for all $x,y\in \mathbb{R}^{d}$,}
\end{equation*}
where $x\cdot y$ denotes the Euclidean scalar product. We claim that a configuration 
$\xi^+ = (z_1,\ldots,z_H)$ satisfying
\begin{equation}\label{eq-171}
\|z_j\|=\dfrac{1}{\sqrt{Hm}},\quad \|z_j-z_l\|=\dfrac{2}{\sqrt{Hm}} \sin \dfrac{\pi}{H},\quad \text{for all $j, l$, with $j\neq l$,}
\end{equation}
namely, a configuration with the $H$ bodies lying at the vertices of a regular $H$-gon inscribed in the circle of radius  
$1/\sqrt{Hm}$ and center the origin, is a (normalized) central configuration for $U$.

Indeed, it is well known that $\xi^+$ is a central configuration for the $H$-body problem, namely
\begin{equation} \label{eq-171bis}
\partial_{x_j} V(\xi^+) = - m \, V(\xi^+)z_j,\quad j=1,\ldots,H,
\end{equation}
where $V$ is the $H$-body potential. A simple computation shows that
\begin{equation} \label{eq-172}
V(\xi^+) = \dfrac{H^{\frac{3}{2}} (H-1) m^{\frac{5}{2}}}{4 \sin \frac{\pi}{H}} ,\qquad U(\xi^+)=(Hm)^{\frac{3}{2}}+ \dfrac{H^{\frac{3}{2}} (H-1) m^{\frac{5}{2}}}{4 \sin \frac{\pi}{H}}.
\end{equation}
Taking into account \eqref{eq-171bis}, this implies that
\begin{equation*}
\partial_{x_j} U(\xi^+)=-\dfrac{m}{\|z_j\|^3} z_j + \partial_{x_j} V(\xi^+)=-m \, (Hm)^{\frac{3}{2}} z_j - m \, V(\xi^+)z_j =-m \,U(\xi^+)z_j,
\end{equation*}
proving that $\xi^+$ is a central configuration for the potential $U$, as well.

As far as condition $(\textsc{BS})$ is concerned, let us first recall that
\begin{equation*}
\langle \nabla^2 \mathcal{U}(\xi^+)y,y \rangle = \langle \nabla^2 U(\xi^+)y,y \rangle +U(\xi^+)|y|^2
\end{equation*}
(see \eqref{Moeckel}); moreover, we observe that
\begin{equation*}
\langle \nabla^2 U(\xi^+)y,y \rangle =(D^2 U(\xi^+)y)\cdot y,
\end{equation*}
where $D^2 U$ is the Hessian matrix of $U$. Now, a standard computation shows that
\begin{align*}
(D^2 U(\xi^+)y)\cdot y 
&= m\, \sum_{j=1}^{H} \left(-\dfrac{\|y_j\|^2}{\|z_j\|^3}+3\dfrac{(z_j\cdot y_j)^2}{\|z_j\|^5}\right)+
\\ & \quad + m^2\, \sum_{l<j} \left(-\dfrac{\|y_l-y_j\|^2}{\|z_l-z_j\|^3}+3\dfrac{\bigl{(}(z_l-z_j)\cdot (y_l-y_j)\bigr{)}^2}{\|z_l-z_j\|^5}\right).
\end{align*}
Taking into account \eqref{eq-171} we deduce that
\begin{align*}
(D^2 U(\xi^+)y)\cdot y 
&= - H^{\frac{3}{2}} m^{\frac{5}{2}}  \, \|y\|^2 
+ 3 H^{\frac{5}{2}} m^{\frac{7}{2}} \sum_{j=1}^{H} (z_j\cdot y_j)^2
- \dfrac{H^{\frac{3}{2}} m^{\frac{7}{2}} }{ 8 \sin^{3} \frac{\pi}{H}}\, \sum_{l<j} \|y_l-y_j\|^2
\\
& \quad
+ \dfrac{3  H^{\frac{5}{2}} m^{\frac{9}{2}} }{32 \sin^{5} \frac{\pi}{H}} \, \sum_{l<j}  \bigl{(} (z_l-z_j)\cdot (y_l-y_j) \bigr{)}^2.
\end{align*}
Hence, we have
\begin{align*}
&(D^2 U(\xi^+)y)\cdot y + m \, U(\xi^+)\|y\|^2
\\
&= 3 H^{\frac{5}{2}} m^{\frac{7}{2}} \sum_{j=1}^{H} (z_j\cdot y_j)^2
- \dfrac{H^{\frac{3}{2}} m^{\frac{7}{2}} }{ 8 \sin^{3} \frac{\pi}{H}}\, \sum_{l<j} \|y_l-y_j\|^2
\\
& \quad
+ \dfrac{3 H^{\frac{5}{2}} m^{\frac{9}{2}}}{32 \sin^{5} \frac{\pi}{H}} \, \sum_{l<j}  \bigl{(} (z_l-z_j)\cdot (y_l-y_j) \bigr{)}^2
+\dfrac{H^{\frac{3}{2}}(H-1) m^{\frac{7}{2}} }{4 \sin \frac{\pi}{H}}\,\|y\|^2.
\end{align*}
Using the Lagrange multiplier method, it is possible to show that
\begin{equation*}
\sum_{j=1}^{H}  (z_j\cdot y_j)^2=\dfrac{1}{Hm}\|y\|^2,
\qquad
\sum_{l<j} \|y_l-y_j\|^2\leq \dfrac{H(H-1)}{2}\|y\|^2,
\end{equation*}
for every $y\in \mathbb{R}^{d}$. Hence, we obtain
\begin{align}
&(D^2 U(\xi^+)y)\cdot y + m \, U(\xi^+)\|y\|^2
\\
&\geq  
3 H^{\frac{3}{2}} m^{\frac{5}{2}} \|y\|^{2} + \dfrac{H^{\frac{3}{2}}(H-1) m^{\frac{7}{2}}}{16 \sin^{3} \frac{\pi}{H}} \biggl{(} 4\sin^{2} \frac{\pi}{H} - H \biggr{)} \|y\|^{2},
\label{eq-6.9}
\end{align}
for every $y\in \mathbb{R}^{d}$.

If $H=2$ and $H=3$, we immediately deduce that $(D^2 U(\xi^+)y)\cdot y + m \, U(\xi^+)\|y\|^2 > 0$ for $y \neq 0$ and thus the $(\textsc{BS})$-condition is satisfied (even more, $\xi^+$ is a locally minimizing central configuration).

Let us now deal with the case $H\geq 4$. From
\begin{equation*}
-\dfrac{1}{8} U(\xi^+) m \|y\|^{2} = - \dfrac{1}{8} H^{\frac{3}{2}} m^{\frac{5}{2}} \|y\|^{2} - \dfrac{H^{\frac{3}{2}}(H-1) m^{\frac{7}{2}}}{32 \sin \frac{\pi}{H}} \|y\|^{2},
\end{equation*}
for every $y\in \mathbb{R}^{d}$, and from inequality \eqref{eq-6.9}, we deduce that the $(\textsc{BS})$-condition is satisfied if 
\begin{equation*}
\dfrac{25}{8} H^{\frac{3}{2}} m^{\frac{5}{2}} + \dfrac{H^{\frac{3}{2}}(H-1) m^{\frac{7}{2}}}{32 \sin^{3} \frac{\pi}{H}} \biggl{(} 9 \sin^{2} \frac{\pi}{H} - 2H \biggr{)}> 0.
\end{equation*}
A simple computation shows that this is true if and only if
\begin{equation*}
m < \dfrac{100 \sin^{3} \frac{\pi}{H}}{(H-1) \Bigl{(} 2H - 9 \sin^{2} \frac{\pi}{H} \Bigr{)}}.
\end{equation*}
Summing up, the regular $H$-gon with $H \geq 4$ satisfies the $(\textsc{BS})$-condition for small values of the masses. 
This is a rather unexpected fact: indeed, in \cite[Section~5.2]{BaSe-08} it is proved, for the $H$-body problem with $H \geq 4$, the regular $H$-gon never satisfies the $(\textsc{BS})$-condition. Following the computation of this section, we actually realize that the presence of a 
Keplerian part in the potential $U$ plays an essential role.

\bibliographystyle{elsart-num-sort}
\bibliography{BDFT-biblio}

\end{document}